\documentclass[12pt]{article}
\providecommand{\keywords}[1]
{
  \small	
  \textbf{\textit{Keywords---}} #1
}

\usepackage{lipsum}
\newcommand\blfootnote[1]{%
  \begingroup
  \renewcommand\thefootnote{}\footnote{#1}%
  \addtocounter{footnote}{-1}%
  \endgroup
}

\usepackage[backend=biber, style=alphabetic ]{biblatex}

\usepackage[hidelinks]{hyperref}
 
\usepackage{amsmath, amsthm, amssymb, latexsym,epsfig,amsthm,enumerate,multicol,wasysym}
\usepackage{color}
\usepackage{graphicx}
\usepackage{nccrules}
\usepackage{xfrac}
\usepackage{hyperref}
\usepackage{textcomp}
\usepackage{tikz}
\usepackage{xcolor}
\usepackage{comment}

\numberwithin{equation}{section}

\theoremstyle{plain}
\newtheorem{theorem}{Theorem}[section]
\newtheorem{lemma}[theorem]{Lemma}
\newtheorem{corollary}[theorem]{Corollary}
\newtheorem{proposition}[theorem]{Proposition}

\theoremstyle{definition}
\newtheorem{definition}[theorem]{Definition}

\newtheorem{claim}[theorem]{Claim}

\newtheorem{case[theorem]}{Case}

\theoremstyle{remark}
\newtheorem{remark}[theorem]{Remark}

\numberwithin{equation}{section}

\def\R{\mathbb{R}}
\def\C{\mathbb{C}}
\usepackage[normalem]{ulem}

\newcommand{\Z}{\mathbb Z}
\newcommand{\N}{\mathbb N}
\newcommand{\cW}{\mathcal W}
\newcommand{\cJ}{\mathcal J}
\newcommand{\cB}{\mathcal B}
\newcommand{\cK}{\mathcal K}
\newcommand{\cH}{\mathcal H}
\newcommand{\cL}{\mathcal L}
\newcommand{\cM}{\mathcal M}
\newcommand{\cI}{\mathcal I}

\newcommand{\cE}{\mathcal E}
\newcommand{\cT}{\mathcal T}

\newcommand{\bL}{\mathbf L}
\newcommand{\bk}{\mathbf k}

\newcommand{\supp}{\mathrm{supp}}
\newcommand{\dist}{\mathrm{dist}}
\newcommand{\diam}{\mathrm{diam}}

\def\R{\mathbb{R}}

\title{On the Eigenvalue Distribution of Spatio-Spectral Limiting Operators in Higher Dimensions}
 \author{Arie Israel and Azita Mayeli}
\date{}

\begin{document}

\maketitle

\blfootnote{A. Israel was supported by the Air Force Office of Scientific Research, under award FA9550-19-1-0005.  A. Mayeli was supported in part by AMS-Simons Research Enhancement Grant and the PSC-CUNY research grants.}

\begin{abstract}
Prolate spheroidal wave functions are an orthogonal family of bandlimited functions on $\R$ that have the highest concentration within a specific time interval. They are also identified as the eigenfunctions of a time-frequency limiting operator (TFLO), and the associated eigenvalues belong to the  interval $[0, 1]$. Previous work has studied the asymptotic distribution and clustering behavior of the TFLO eigenvalues.

In this paper, we extend these results to multiple dimensions. We prove estimates on the eigenvalues of a \emph{spatio-spectral limiting operator} (SSLO) on $L^2(\R^d)$, which is an alternating product of projection operators associated to given spatial and frequency domains in $\R^d$. If one of the domains is a hypercube, and the other domain is a convex body satisfying a symmetry condition, we derive quantitative bounds on the distribution of the SSLO eigenvalues in the interval $[0,1]$. 

To prove our results, we design an orthonormal system of wave packets in $L^2(\mathbb{R}^d)$ that are highly concentrated in the spatial and frequency domains. We show that these wave packets are ``approximate eigenfunctions'' of a spatio-spectral limiting operator. To construct the wave packets, we use a variant of the Coifman-Meyer local sine basis for $L^2[0,1]$, and we lift the basis to higher dimensions using a tensor product.
\end{abstract} 


\tableofcontents 
 
\section{Introduction} 
The Heisenberg uncertainty principle states that a function $f : \R^d \rightarrow \R$ and its Fourier transform $\widehat{f} : \R^d \rightarrow \R$ cannot simultaneously be well-localized in the spatial and frequency domains.  
As a result, one is curious about how much a bandlimited function can be concentrated on a given spatial  region, and how much a spacelimited function can be concentrated on a given frequency region. 
We refer to this as the {\it spatio-spectral concentration problem}. We shall see that this problem can be formulated as the eigenvalue problem for a certain self-adjoint operator on $L^2(\R^d)$.

%\footnote{Azita: I changed the name of the problem in $\R^d$ to ``spatio-spectral concentration problems'' to handle both cases of exactly spacelimited/approximately bandlimited functions and exactly bandlimited/approximately spacelimited functions. This name seems to be standard in the literature, e.g., literature  search pulls a few results (e.g., see  \href{https://arxiv.org/abs/math/0408424}{{\color{blue} link 1}} and \href{https://arxiv.org/ftp/arxiv/papers/1608/1608.03031.pdf}{{\color{blue} link 2}} and \href{https://www.sciencedirect.com/science/article/pii/S1063520315000470}{{\color{blue} link 3}}), whereas ``space concentration problem'' obscures the Fourier constraints in the problem, and I could not find this name appearing in the literature. \\

%Arie: I agree with what you said. Thanks for looking at this carefully. I also found these papers \href{https://link.springer.com/content/pdf/10.1007/s43670-021-00014-2.pdf}{{\color{blue} link 4}} and \href{https://arxiv.org/pdf/2110.15523.pdf}{{\color{blue} link 5}}
%where the operator is introduced as spatio-spectral limiting operator, while we defined spatio-spectral limiting operator. The authors say that the name spatio-spectral is known in geometry context.  Should we use that name? Our operator will be SSLO instead of SFLO. We can also refer to their paper, if you agree.}

By the seminal work of Landau, Slepian, and Pollak, \cite{Bell1,Bell2,Bell3}, it is known that the \emph{prolate spheroidal wave functions} (PSWFs)  are the solutions to the  spatio-spectral concentration problem in dimension $d=1$, for the case when the spatial and frequency domains are intervals $I, J$ in $\R$. Specifically, the PSWFs $\{\psi_k\}$ may be defined recursively as follows: First, $\psi_1$ is a bandlimited unit-norm $L^2$ function with the largest proportion of its $L^2$ energy on $I$; then, after $\psi_1,\cdots,\psi_k$ are determined, $\psi_{k+1}$ is a bandlimited unit-norm $L^2$ function that is orthogonal to $\psi_1,\cdots,\psi_k$ and with the largest proportion of its $L^2$ energy on $I$.  The PSWFs can be characterized as the eigenfunctions of a compact self-adjoint  operator, and accordingly, they form an orthogonal basis for the space of all bandlimited functions with Fourier support in $J$. The  PSWF eigenvalues $\{\lambda_k\}$ measure the (spatial) $L^2$ concentration of the PSWF sequence on $I$, i.e., $\lambda_k = \int_I |\psi_k(x)|^2 dx/\int_\R |\psi_k(x)|^2dx$. Naturally, these numbers belong to the interval $(0,1]$. Further, the $\lambda_k$ can be interpreted as the eigenvalues of a compact self-adjoint operator, and so, they form a  monotonic sequence converging to $0$. 

The PSWF eigenvalues exhibit a sharp clustering behavior near $0$ and $1$ and their distribution properties reveal information about the dimension of the space of approximately spacelimited and exactly bandlimited functions on the real line.  This behavior has been studied in the  papers by Landau, Slepian, and Pollak \cite{Bell1,Bell2, Bell3}. Further results and improvements were obtained more recently; e.g., see Osipov \cite{Osipov13}, Israel \cite{Israel15},  Karnik-Romberg-Davenport  \cite{Karnik21}, and Bonami-Jaming-Karoui \cite{Bonami21}. The discrete version of the problem has also been considered  \cite{Bell5, Karnik21}.

In higher dimensions, the spatio-spectral concentration problem has been examined in \cite{Bell4, Shkolnisky07}  for the case when the space and frequency domains are discs in $\R^2$. Then the solution sequence can be described in terms of the eigenfunctions of a certain differential operator. These functions form an orthogonal system in $L^2(\R^2)$, known as \emph{generalized prolate spheroidal functions} (GPSFs). We should remark that the higher-dimensional setting is relevant for applications in scientific imaging problems, e.g., in cryoelectron microscopy (cryo-EM) and MRI,  which rely on representations with respect to an orthogonal basis of approximately space-limited and (exactly) bandlimited functions of several variables. In particular, GPSFs have been used to perform signal compression in MRI  \cite{Yang02}, clustering and principal component analysis for cryo-EM imaging data \cite{Shkolnisky17A, Shkolnisky17B} and also for the analysis of the alignment problem in cryo-EM \cite{Lederman17}; see also \cite{LedermanSinger17, LedermanSinger20}.

By now, many of these problems are rather well understood in one dimension. However, in the higher-dimensional case, fundamental questions still remain unresolved. 
%In particular, no quantitative bounds on the GPSF eigenvalues are presently known.\footnote{The paper \cite{GreengardSerkh18} presents a numerical method to compute the eigenvalues associated to GPSFs, and evaluates the eigenvalues numerically in several examples. An asymptotic formula for the distribution of the GPSF eigenvalues, due to Landau \cite{Landau75}, will be stated in the next section -- see \eqref{landau-eig-clustering_eqn}.}   
The purpose of this work is to address this gap in understanding between $d=1$ and $d \geq 2$. 
 % The current work serves as a \sout{first} step to address this gap in understanding between $d=1$ and $d \geq 2$.
When both the spatial and frequency domains are balls in $\R^d$, the orthogonal family consisting of the GPSFs in $\R^d$ has been studied in \cite{Lederman17,GreengardSerkh18}, including several applications to numerical analysis for processing of bandlimited functions (e.g., quadrature, differentiation, and interpolation schemes). The main question addressed in this paper is to explain \emph{precisely how well} such functions are concentrated in the spatial domain.

In this paper, we give a solution to a related problem: We derive bounds on the distribution of eigenvalues for the spatio-spectral concentration problem, in the case when the spatial domain is a hypercube in $\R^d$, and the frequency domain is a convex body in $\R^d$ satisfying a certain symmetry condition. In particular, our analysis applies to the case when the spatial and frequency domains are a cube and ball in $\R^d$, respectively. 
%We also derive a somewhat sharper estimate in the case when the spatial and frequency domains are both cubes in $\R^d$.

%Following the submission of our paper to the arXiv, we were made aware of the preprint \cite{MaRoSp23} submitted in the same week, which presented related results. For a comparison of the two works, see Remark \ref{comp:rem}.

\subsection{Statement of the variational problem and main results}  

The Fourier transform $\mathcal{F} $ and inverse Fourier transform $\mathcal{F}^{-1} $ on $L^2(\R^d)$ are defined by $\mathcal{F}[f](\xi) = \int_{\R^d} f(x) e^{-ix \cdot \xi} dx$, and  $\mathcal{F}^{-1}[g](x) = \frac{1}{(2 \pi)^d} \int_{\R^d} g(\xi) e^{i x \cdot \xi} d \xi$. We adopt the  shorthand $\widehat{f} = \mathcal{F}(f)$ for the Fourier transform of $f$. 
We write $\chi_A : \R^d \rightarrow \R$ for the indicator function of a set $A \subset \R^d$, given by $\chi_A(x) = 1$ if $x \in A$, $\chi_A(x) = 0$ if $x \notin A$.

Given a measurable set $S \subset \R^d$, a function $f \in L^2(\R^d)$ is said to be \emph{$S$-bandlimited} if $\widehat{f}(\xi) = 0$ for a.e. $\xi \in \R^d \setminus S$. We denote the subspace of all $S$-bandlimited functions in $L^2(\R^d)$ by $\cB(S)$.

Given measurable sets  $Q, S \subset \R^d$,   with finite and positive $d$-dimensional Lebesgue measures $\mu_d(Q),\mu_d(S) \in (0,\infty)$,  we pose the spatio-spectral concentration problem:  {\it Which $S$-bandlimited functions have the   most  $L^2$-energy   on  $Q$?}
More precisely, we consider the variational problem
\begin{align}\label{Raleigh-quo}
\sup_{\psi \in \cB(S) \setminus \{0\}}  \left\{ \mathcal{J}(\psi) := \| \psi \|_{L^2(Q)}^2 \big/ \| \psi \|_{L^2(\R^d)}^2 = \int_Q |\psi(x)|^2 dx\Big/ \int_{\R^d} |\psi(x)|^2 dx  \right\}.
\end{align}

%This is a  variational problem in the intersection of applied harmonic analysis, spectral analysis of operators,  and signal analysis that is  relevant to the study of the space of bandlimited functions in   signal processing. 
Problem \eqref{Raleigh-quo} is equivalent to an eigenvalue problem for a  compact self-adjoint operator on $L^2(\R^d)$. Specifically, define the  orthogonal projection operators $P_Q$ ($Q$-spacelimiting) and $B_S$ ($S$-bandlimiting) on $L^2(\R^d)$ by
$P_Q[f] := \chi_Q f,  $ and $ B_S[f] := \mathcal{F}^{-1}[ \chi_S(\xi) \widehat{f}(\xi)]$, for $f\in L^2(\R^d)$.  We refer to $\widetilde{\cT}_{Q,S}:= B_S P_Q B_S : L^2(\R^d) \rightarrow L^2(\R^d)$ as a {\it spatio-spectral limiting operator} (SSLO). 

The operator  $\widetilde{\cT}_{Q,S}$ is compact and positive semidefinite on $L^2(\R^d)$, and all of its eigenvalues belong to the interval $[0,1]$; see \cite{Landau75}.

One can show that $\mathcal{J}(\psi)$ in \eqref{Raleigh-quo} is the Rayleigh quotient of $\widetilde{\cT}_{Q,S}$, in the sense that $\mathcal{J}(\psi) = \langle \psi , \widetilde{\cT}_{Q,S} \psi \rangle/ \| \psi \|^2$  for $\psi \in \mathcal{B}(S)$, $\psi \neq 0$; see Remark \ref{remk:SSLO_alt}. Thus, problem \eqref{Raleigh-quo} is equivalent to the task of determining the largest eigenvalue of $\widetilde{\cT}_{Q,S}$.

More generally, for $k\geq 0$, we  consider the task of maximizing the quantity $\inf_{\psi \in H, \psi \neq 0} \mathcal{J}(\psi)$ over all $(k+1)$-dimensional subspaces $H \subset \cB(S)$. By the variational characterization of eigenvalues, the optimal value for this problem is equal to the $k$'th largest eigenvalue of $\widetilde{\cT}_{Q,S}$, and the optimal subspace $H$ is spanned by the top $k$ eigenfunctions of $\widetilde{\cT}_{Q,S}$. We will denote ${\lambda}_k(Q,S) \in (0,1]$ for the $k$'th largest  eigenvalue of $\widetilde{\cT}_{Q,S}$ (counted with multiplicity), and  ${\psi}_k(Q,S) \in L^2(\R^d)$ for the corresponding eigenfunction of $\widetilde{\cT}_{Q,S}$, so that $\{ \psi_k(Q,S)\}_{k \geq 0}$ is an orthonormal basis for $\cB(S) = \ker (\widetilde{T}_{Q,S})^\perp = \mathrm{range}(\widetilde{T}_{Q,S})$.
  
We will investigate the distribution and decay rate of the eigenvalue sequence $\lambda_k(Q,S)$. Specifically,  given $\epsilon \in (0,1/2)$, we ask to estimate the number  of eigenvalues in the interval $(\epsilon, 1]$,  and in  the interval $(\epsilon, 1-\epsilon)$. Further, we ask to determine the decay rate of  $\lambda_k(Q,S)$ as $k \rightarrow \infty$.

If $Q$ and $S$ are bounded domains in $\R^d$, it is possible to prove a decay estimate of the form $\lambda_k(Q,S) \leq C_d \exp(-c(Q,S) k^{1/d})$, for a dimensional constant $C_d > 0$, and a constant $c(Q,S) > 0$ determined by the diameters of $Q$ and $S$. The paper \cite{Karnik21} proves such an estimate in the $d=1$ case -- see \eqref{1d_eigdecay}. For a result in higher dimensions, see Corollary \ref{eig_decay:cor}. It is interesting to ask if the exponent $1/d$ in this bound can be improved to some $\alpha > 1/d$. In dimension $d=1$, the bound is known to not be asymptotically sharp, by the work of \cite{Bonami21}.

We are interested in estimating the eigenvalues $\lambda_k(Q,S)$ in the regime when $\mu_d(Q)\cdot \mu_d(S) \rightarrow \infty$. We will force the domains to grow in a controlled (isotropic) manner. Specifically, we shall fix the spatial domain $Q$ and rescale the frequency domain by a factor of $r$, i.e., we shall replace $S$ by $S(r) := rS = \{r x : x \in S\}$. We write $\lambda_k(Q,S(r))$ to denote the sequence of nonzero eigenvalues of the operator $B_{S(r)}P_QB_{S(r)}$. For $\epsilon \in (0,1)$, let $M_\epsilon(r)$ denote the number of eigenvalues $\lambda_k(Q,S(r))$ that are larger than $\epsilon$. If $Q$  and $S$ are measurable sets in $\R^d$, then Landau \cite{Landau75} proves that
\begin{equation}
\label{landau-eig-clustering_eqn}
\lim_{r \rightarrow \infty} r^{-d} M_\epsilon(r)  = (2 \pi)^{-d} \mu_d(Q) \mu_d(S).
\end{equation}
The techniques in \cite{Landau75} do not yield a quantitative convergence rate in \eqref{landau-eig-clustering_eqn}. The primary goal of this paper is establish a quantitative version of \eqref{landau-eig-clustering_eqn} under suitable assumptions on the domains $Q,S$. 

Below in Section \ref{sec:main-results}, we discuss our contributions and compare our results to those of the recent work \cite{MaRoSp23}. In Section \ref{subsec:1dresults}, we will elaborate  on the prior results for time-frequency limiting operators in one dimension.

\subsubsection{Contributions} \label{sec:main-results}  

Let $Q, S\subset \R^d$ be measurable sets with finite and positive $d$-dimensional Lebesgue measures, $\mu_d(Q), \mu_d(S) \in (0,\infty)$. Denote the \emph{$r$-isotropic dilation} of $S$ by $S(r) = rS = \{ r x : x \in S \} \subset \R^d$, for $r > 0$. Consider the positive semidefinite operator $\widetilde{\cT}_{Q,S(r)} := B_{S(r)} P_Q B_{S(r)}$ on $L^2(\R^d)$, which is the SSLO associated to the space and frequency domains $Q$ and  $S(r)$. 
Let $\{\lambda_k(Q,S(r))\}_{k \geq 0}$ in $(0,1]$ be the non-increasing arrangement of the positive eigenvalues of $\widetilde{\cT}_{Q,S(r)}$ (counted with multiplicity). Note that $\{\lambda_k(Q,S(r))\}_{k \geq 0}$ is also the sequence of positive eigenvalues of the operator $\cT_{Q,S(r)} := P_Q B_{S(r)} P_Q$ -- see Remark \ref{remk:SSLO_alt}.

 In our first theorem, we let $Q = [0,1]^d$ be the unit hypercube, and let $S$ be a convex domain in $\R^d$ satisfying a symmetry condition. We say that a convex set $S \subset \R^d$ is \emph{coordinate-wise symmetric} provided that
\begin{equation}\label{eqn:coord_sym}
 (x_1,\cdots,x_d) \in S, \; \sigma = (\sigma_1,\cdots,\sigma_d) \in \{\pm 1\}^d \implies (\sigma_1 x_1,\cdots, \sigma_d x_d) \in S.
\end{equation}
We write $B(x,r) \subset \R^d$ for the Euclidean ball centered at $x \in \R^d$ of radius $r > 0$

\begin{theorem}\label{mainthm:cube_convex}
Let $Q = [0,1]^d$ be the standard unit hypercube, and let $S \subset B(0,1)$ be a compact coordinate-wise symmetric convex set in $\R^d$. 

There exists a dimensional constant $C_d > 0$ (independent of $S$) such that, for any $\epsilon \in (0,1/2)$ and $r \geq 1$, 
\begin{align}\label{eig_clust:eqn1}
&| \# \{ k : \lambda_k(Q,S(r)) > \epsilon\} - (2\pi)^{-d} \mu_d(S(r)) | \leq C_d E_d(\epsilon,r) ,\\
\label{eig_clust:eqn2}
& \# \{ k : \lambda_k(Q,S(r)) \in (\epsilon, 1-\epsilon)\} \leq C_d E_d(\epsilon,r), \\
\label{eig_clust:eqn3}
&E_d(\epsilon,r) := \max \{ r^{d-1} \log(r/\epsilon)^{5/2}, \log(r/\epsilon)^{5d/2} \}.
\end{align}
\end{theorem}

The bounds \eqref{eig_clust:eqn1} and \eqref{eig_clust:eqn2} describe the clustering behavior of SSLO eigenvalues near $0$ and $1$. Observe that $(2\pi)^{-d} \mu_d(S(r)) = C_d \mu_d(S) r^d$, while the term on the right-hand side of \eqref{eig_clust:eqn1} and \eqref{eig_clust:eqn2} is of lower order $\sim e(r) := r^{d-1} \log(r)^{5/2}$ for fixed $\epsilon>0$ and large $r$. Thus, Theorem \ref{mainthm:cube_convex} implies that the first $C_d \mu(S) r^d - O(e(r))$ many eigenvalues $\lambda_k(Q,S(r))$ will be very close to $1$, followed by at most $O(e(r))$ many eigenvalues of intermediate size, and the remaining eigenvalues decay to $0$ at a near-exponential rate (see Corollary \ref{eig_decay:cor} for a  statement about the decay rate). In particular, \eqref{eig_clust:eqn1} gives a quantitative version of Landau's formula \eqref{landau-eig-clustering_eqn}.

The main idea in the proof of Theorem \ref{mainthm:cube_convex} is to construct an ``approximate eigenbasis'' for the operator $\cT_r = \mathcal{T}_{Q,S(r)}$. More specifically, we shall define a family of ``wave packet'' functions $\{\psi_\nu \}$, which are highly localized in both the spatial and Fourier domains, and which form an orthonormal basis for $L^2(Q) = \ker(\cT_r)^\perp$, with the property that either $\| \cT_r \psi_\nu\|$ or $\| \cT_r \psi_\nu - \psi_\nu \|$ is small for all but a few indices $\nu$. We construct this orthonormal wave packet basis in Section \ref{sec:lsb}. We present a key functional analysis result in Lemma \ref{func_anal:lem}, and we prove the essential properties of the basis $\{\psi_\nu\}$ in Proposition \ref{main:prop}. Using these ingredients, we complete the proof of Theorem \ref{mainthm:cube_convex} in Section  \ref{proof-of-mainthm}.

\begin{remark}\label{imp:rem} 
The  construction of the wave packet basis involves the choice of a cutoff function  $v(x)$  in \eqref{c0}.  By choosing the cutoff function instead as  $v_m(x) := e^{-(1-x^2)^{-m}}$ for $x\in (-1,1)$ and zero elsewhere,  one  can improve  the exponent $5/2$ in \eqref{eig_clust:eqn1}-\eqref{eig_clust:eqn3} to $2+\alpha$,  for $\alpha = 1/m$, and any integer $m\geq 2$, with the constant $C_d$ replaced by a constant depending on $d$ and $m$. For more details, see Remark \ref{imp:detailed rem}.
 \end{remark}

\begin{remark}
    \label{comp:rem}
After posting this work to the arXiv, we learned about the paper \cite{MaRoSp23} that was released in the same week as ours. There, the authors study the distribution of the eigenvalues $\lambda_k(Q,S)$ for domains $Q$ and $S$ in $\R^d$, such that the boundaries $\partial Q$ and $\partial S$ satisfy an Ahlfors regularity condition, and their Hausdorff $(d-1)$-measures satisfy  $|\partial Q| \cdot |\partial S| \geq 1$. Under this assumption, the authors prove an estimate of the following form: For every $\alpha \in (0,1/2]$ there exists $A_{\alpha, d} \geq 1$ such that for all $\epsilon \in (0,1/2)$,
\begin{equation}\label{Romero_bd}
\# \{k \in \N : \lambda_k(Q,S) \in (\epsilon,1-\epsilon)\} \leq A_{\alpha,d} \frac{|\partial Q|}{\kappa_{\partial Q}} \frac{|\partial S|}{\kappa_{\partial S}} \log \left( \frac{|\partial Q| |\partial S|}{\kappa_{\partial Q} \epsilon}\right)^{2d(1+\alpha) + 1}.
\end{equation}
The authors also prove an estimate on the number of eigenvalues larger than $\epsilon$, which they show to be proportional to $\mu_d(Q)\mu_d(S)$ up to an additive error term bounded by the right-hand side of \eqref{Romero_bd}. Here, $\kappa_{\partial Q}$ and $\kappa_{\partial S}$ are constants that describe the Ahlfors regularity of the boundaries of $Q$ and $S$. In particular, if $Q$ and $S$ are fixed domains, and one replaces $S$ by its isotropic dilation $S(r)$ in \eqref{Romero_bd}, one gets the bound $\# \{ k : \lambda_k(Q,S(r)) \in (\epsilon, 1-\epsilon)\} = O_{\alpha,d}(r^{d-1} \log(r/\epsilon)^{2d(1+\alpha) + 1})$, where the constant in the $O_{\alpha,d}$ notation depends on $\alpha$, $d$, $Q$, and $S$. For comparison, when $Q = [0,1]^d$ and $S$ is convex and coordinate-wise symmetric, Theorem \ref{mainthm:cube_convex} gives an upper bound by $O_d(\max \{ r^{d-1} \log(r/\epsilon)^{5/2}, \log(r/\epsilon)^{(5/2)d}\})$, which is  sharper in the regime $\epsilon \sim r^{-s}$ and $r$ large.
%This improvement can be traced to the favorable geometric inequalities for counting lattice points in convex bodies, which we develop in Section \ref{sec:lattice}, and apply within the proof of Proposition \ref{main:prop}.
\end{remark}

The bounds in Theorem \ref{mainthm:cube_convex} can be somewhat improved when the space and frequency limiting domains are cubes in $\R^d$ with sides parallel to the coordinate axes. This is  our  next  result. 
 
%Furthermore,  we  show  that  the estimate  of  the number of   eigenvalues in    $N_\epsilon$, related to the ``transition" region $(\epsilon, 1-\epsilon)$,   is the higher dimensional version of the  result in $1$D in  \cite{Karnik21}.  For convenience,  we have collected the  result of \cite{Karnik21}   in this paper  in  Theorem \ref{Karnik}.  

\begin{theorem}\label{tensor:propA}

Let $Q = [0,1]^d$ and $S=[-1,1]^d$, so that $S(r) = [-r,r]^d$. Then there exists a dimensional constant $C_d$, such that for any $r \geq 2 \pi$
and $\epsilon \in (0,1/2)$,
\begin{equation}\label{M_eps1:eqn}
\begin{aligned}
&|\# \{ k : \lambda_k([0,1]^d,[-r,r]^d) > \epsilon \} - (r/\pi)^d | \leq C_d  B_d(\epsilon,r),\\
&\# \{ k : \lambda_k([0,1]^d,[-r,r]^d) \in (\epsilon, 1-\epsilon)\} \leq C_dB_d(\epsilon,r),\\
&B_d(\epsilon,r) := \max \{ r^{d-1} \log(r) \log(1/\epsilon), (\log(r) \log(1/\epsilon))^d\}.
\end{aligned}
\end{equation}
\end{theorem}

To prove Theorem \ref{tensor:propA}, we note that  the operator $\widetilde{\cT}_{Q,S}$ on $L^2(\R^d)$ is the $d$-fold composition of one-dimensional  time-frequency  limiting operators acting on each variable. Estimates on the eigenvalues of  $\widetilde{\cT}_{Q,S}$ are easily derived from the corresponding estimates for one-dimensional operators, using the results of \cite{Karnik21} (which we recall in Section \ref{subsec:1dresults}). Details for the proof are given in Section \ref{proof-of-propA}. 

%DELETING THIS COMMENT BECAUSE IT IS NOT TRUE THAT N_epsilon DEPENDS APPROXIMATELY ON W^d
%The results  in \eqref{M_eps1:eqn}  show  that the size of the eigenvalue  sets  in $(\epsilon, 1)$ and $(\epsilon, 1-\epsilon), i.e., M_\epsilon$ and $N_\epsilon$ respectively, depend approximately on $W^d$. 
% \footnote{It is not true that both quantities depend approximately linearly on $W^d$.}
%
%\footnote{\color{blue}  We observed that     $|M_\epsilon - N_\epsilon -(W/\pi)^d|\leq 2C_dB_d$. - I forgot what we talked about this difference. Are we interested in the difference when $\epsilon \to 1/2$? because when  $\epsilon$ is near zero, then we already know the answer, and it is a boring case. {\color{red} I also forgot why we spoke about the difference. Maybe it is not important.}}

%\footnote{questions/notes to myself: I am curious,  are the eigenfunctions of SSLO also orthogonal on $\R^d$, if we choose special $S$? Check the disk case. \\
%We need to clarify that when we are talking about the eigenvalues, we are considering the operator on $B(S)$. I feel we did not clarify it anywhere. {\color{red} The eigenfunctions are orthogonal in $L^2(\R^d)$, because the operator is self-adjoint mapping from $L^2(\R^d) \rightarrow L^2(\R^d)$. Why do you think that the eigenfunctions belong to $B(S)$?}  }

Our final result is a corollary of Theorem \ref{tensor:propA}. It gives a bound on the decay rate of the eigenvalues $\lambda_k(Q,S)$ as $k \rightarrow \infty$. We present the proof in Section \ref{decay-rate:eign}. 

Below, we write $\diam_\infty(A)$ for the diameter of a set $A \subset \R^d$ with respect to the $\ell^\infty$ metric on $\R^d$.

\begin{corollary}
\label{eig_decay:cor}
Let $Q$ and $S$ be compact sets and let $\Delta = \diam_\infty(Q) \cdot \diam_\infty(S) \in (0,\infty)$.  Then
\[
\lambda_{k}(Q,S) \leq C_d \exp \left(-c(\Delta) k^{1/d} \right), \;\; \mbox{for } k \geq 1,
\]
for constants $C_d > 0$, determined by $d$, and $c(\Delta) > 0$, determined by $d$ and $\Delta$.
\end{corollary}

%\begin{remark} \label{rem:dilation-translatoin-inva} 
%Notice that by rescaling and translation of $(Q,S) \mapsto (\Delta Q, \Delta^{-1} S+\alpha)$, the operators $\widetilde{\cT}_{Q,S}$ and $\widetilde{\cT}_{\Delta Q, \Delta^{-1}S+\alpha}$ are similar and the set of their eigenvalues are same. Therefore,  without loss of generality,  in Corollary  \ref{eig_decay:cor}  we may assume that one of the sets has $\ell^\infty$-diameter $\leq 1$.
%\end{remark}

\subsubsection{Related work on time-frequency limiting operators}
\label{subsec:1dresults}

Let $I=[0,1]$ and $J=[-W,W]$, $W>0$.  The %eigenfunctions of the 
eigenvalues of the time-frequency limiting operator $\widetilde{\cT}_W:=\widetilde{\cT}_{I,J} = B_J P_I B_J$ will be denoted $\lambda_k(W)= \lambda_k(I,J) \in (0,1]$, $k \geq 0$, and are referred to as PSWF eigenvalues.  The operator $\widetilde{\cT}_W$ is a compact integral operator with square-integrable continuous kernel. As the eigenvalues of a compact operator, $\lambda_k(W)$ decreases monotonically to $0$ as $k \rightarrow \infty$ for fixed $W > 0$. By Mercer's theorem, $\sum_{k=0}^\infty  \lambda_k(W) = \mathrm{tr}(\widetilde{\cT}_{W})= W/\pi$; see \cite{Landau75}.
 
What is more interesting is that the eigenvalues $\lambda_k(W)$ are clustered near $0$ or $1$, in the following sense: For any $\epsilon \in (0,1/2)$, slightly fewer than $W/\pi$ of the eigenvalues lie in $[1- \epsilon,1]$, far fewer eigenvalues lie in $(\epsilon, 1 - \epsilon)$, and the rest lie in $(0, \epsilon]$. More precisely, Landau \cite{Landau93} proved the following: For $W \geq 2\pi$, let $N = \lfloor W/\pi \rfloor$ be the integer part of $W/\pi$.  Then
 \begin{align}\label{landau-half} 
 \lambda_{N-1}(W) \geq 1/2 \geq \lambda_{N+1}(W).
 \end{align} 
Further, Landau-Widom \cite{LandauWidom80}  proved that  for fixed $\epsilon \in (0,1/2)$, as $W \rightarrow \infty$,
\begin{align}\label{landau-widom}
\# \{ k  \geq 0 : \lambda_k(W) \in (\epsilon, 1 - \epsilon) \} = \frac{2}{\pi^2} \log(W) \log\left( \frac{1}{\epsilon} - 1 \right) + o ( \log(W)).
\end{align} 
We point out that \eqref{landau-widom} is an asymptotic result, not a quantitative one. Indeed, the error term  $o(\log W)$ is not explicitly bounded in \cite{LandauWidom80}. Therefore, from \eqref{landau-widom} it is unclear how the quantity $\# \{ k : \lambda_k(W) \in (\epsilon, 1 - \epsilon) \}$ grows as $\epsilon \rightarrow 0$. In particular, \eqref{landau-widom} is helpless for the task of proving the decay of the eigenvalues $\lambda_k(W)$ as $k \rightarrow \infty$.

We refer to $\{k : \lambda_k(W) \in (\epsilon, 1 - \epsilon) \}$ as the ``transition region'' in the spectrum of  $\widetilde{\cT}_W$. Quantitative upper bounds on the size of the transition region were proven by Osipov \cite{Osipov13} and the first-named author \cite{Israel15}. Subsequently,  Karnik et al. \cite{Karnik21} obtained improved quantitative bounds matching the asymptotic formula \eqref{landau-widom}, valid for any $\epsilon \in (0,1/2)$ and $W > 0$. Bonami et al.  \cite{Bonami21} prove estimates which are  stronger than those in \cite{Karnik21} when $\epsilon$ is  smaller than a negative power of $W$.

We now state a remarkable theorem of \cite{Karnik21}, which gives a quantitative upper bound on the width of the transition region achieving the same asymptotic rate as \eqref{landau-widom}:
\begin{theorem}[\cite{Karnik21}, Theorem 3]
\label{Karnik} For any $W>0$ and $\epsilon\in (0,1/2)$,
 \begin{align}\label{Romberg-Thm:3} 
\# \{ k : \lambda_k(W) \in (\epsilon, 1 - \epsilon) \} \leq \frac{2}{\pi^2} \log \left( \frac{50W}{\pi}+25 \right) \log \left( \frac{5}{\epsilon(1-\epsilon)} \right) + 7.
 \end{align}
\end{theorem}

We shall make use of Theorem \ref{Karnik} later, in the proof of Theorem \ref{tensor:propA}.
As noted in \cite{Karnik21}, Theorem \ref{Karnik} implies the exponential decay of $\lambda_k(W)$ for $k \geq \lceil W/\pi \rceil$: 
\begin{equation}\label{1d_eigdecay}
\lambda_{k}(W) \leq 10 \exp \left[ - \frac{k - \lceil W/\pi \rceil - 6}{c_1 \log(W + c_2)}  \right] \qquad k \geq \lceil W/\pi \rceil,
\end{equation}
where the constants $c_1, c_2 > 0$ are explicit, and $\lceil x \rceil$ is the ceiling function of $x$.

\section{Notations and Preliminaries} \label{sec:prelim}

We write $\N$ for the set of non-negative integers, i.e., $\N = \{ 0,1,2,\cdots \}$, and $\#(A)$ for the cardinality of a set $A$.
We write $|x|=|x|_2$ for the Euclidean $\ell^2$-norm of $x$ in $\R^d$, and $|x|_\infty = \max\{ |x_i| : 1 \leq i \leq d \}$  for the $\ell^\infty$-norm of $x$.
We write $\mu_d$ for Lebesgue measure on $\R^d$, and $|I|$ for the length of an interval $I \subset \R$. 
We write $\| f \|_2$ for the $L^2$ norm of a measurable function $f$, and $\| f \|_\infty$ for the $L^\infty$ norm of $f$.
A box $\bL \subset \R^d$ is the $d$-fold Cartesian product of intervals $L_j$, i.e., $\bL = \prod_{j=1}^d L_j$. An axis-parallel hypercube  $Q$ is a box such that all the intervals have equal length,  i.e., $Q = \prod_{j=1}^d I_j$, with $|I_j| = |I_{k}|$ for all $j,k$.
We write $\chi_A(x)$ for the indicator function of a set $A \subset \R^d$. 
%We define $\dist_\infty(A,B) := \inf_{x \in A, y\in B} | x - y |_{\infty}$ for  $A \subset \R^d$   and $B \subset \R^d$.  

%On $L^2(\R^d)$, we define the Fourier transform $\mathcal{F}[f](\xi) = \int_{\R^d} f(x) e^{-i x \cdot \xi} dx$, and  the inverse Fourier transform $\mathcal{F}^{-1}[g](x) = \frac{1}{(2 \pi)^d} \int_{\R^d} g(\xi) e^{i x \cdot \xi} d \xi$. We adopt the  shorthand $\widehat{f} = \mathcal{F}(f)$ for the Fourier transform of $f$. 

%Given a compact set $S \subset \R^d$ with positive and finite measure, a function $f \in L^2(\R^d)$ is said to be $S$-{\it bandlimited} provided that $\widehat{f}(\xi) = 0$ for almost every  $\xi \in \R^d \setminus S$. We write $\cB(S)$ for the space of all $S$-bandlimited functions in $L^2(\R^d)$. These spaces are also known as Paley-Wiener spaces.    

%***OLD VERSION*** Let $T : \cH \rightarrow \cH$ be the the Hilbert-Schmidt on the Hilbert space $\mathcal H$. The Hilbert-Schmidt norm  is   $\| T \|_{HS} = \left( \sum_{j=1}^\infty \| T e_j \|^2 \right)^{1/2}$, where $\{ e_j \}$ is any orthonormal basis for  the   space $\cH$. Here, the norm  $\| T e_j \|$ is in $\mathcal H$. 
%By $\|T\|$ we denote the  operator norm of $T$.

Let $\cH$ be a Hilbert space, with norm $\| \cdot \|$. 
Recall, the Hilbert-Schmidt norm of a linear operator $T : \cH \rightarrow \cH$ is $\| T \|_{HS} = \left( \sum_{j=1}^\infty \| T e_j \|^2 \right)^{1/2}$, where $\{ e_j \}$ is any orthonormal basis for the space $\cH$, and the value of the Hilbert-Schmidt norm does not depend on the choice of the orthonormal basis.  
By $\|T\|$ we denote the  operator norm of $T$.

Write $B(x,r)$ for the Euclidean ball in $\R^d$ centered at $x$ with radius $r$. 
Given  $x \in \R^d$ and $K \subset \R^d$, 
we write $x + K$ for the set $\{ x + y : y \in K \}$. Given $K,K' \subset \R^d$, we write $K \oplus K'$ for the Minkowski sum of $K$ and $K'$, and  $K \ominus K'$ for the Minkowski difference of $K$ and $K'$, defined by:
\[
\begin{aligned}
&K \oplus K' = \{ y + y' : y \in K, \; y' \in K'\} \\
&K \ominus K' =  \{ c \in \R^d : c + K' \subset K \} = \bigcap_{y' \in K'} (-y' + K).  
\end{aligned}
\]
If $K,K'$ are convex then the sum $K \oplus K'$ and difference $K \ominus K'$ are convex as well. (On the second point, recall that the intersection of convex sets is convex.)

When $x=0$, we write $B(r) = B(0,r)$.  Then it is easy to see that 
for any closed subset $K \subset \R^d$,
\[
\begin{aligned}
&K \ominus B(r) = \{ x \in \R^d : x \in K, \;  d(x, \partial K) \geq r \}, \\
&K \oplus B(r) = \{ x \in \R^d : d(x,K) \leq r \}.
\end{aligned}
\]
The set $K \oplus B(r)$ is called the {\it (closed) $r$-neighborhood}  of $K$. Note that $K \oplus B(r) \supset K$ and $K \ominus B(r) \subset K$ for any closed set $K$. If $K$ is convex, then  $K \ominus B(r)$ and $K \oplus B(r)$ are also convex.

\subsection{Basic properties of SSLOs} 
Let $Q$ and $S$ be measurable subsets of $\R^d$, with positive and finite $d$-dimensional Lebesgue measure. We denote $\cB(S) = \{p f \in L^2(\R^d) : \chi_{\R^d \setminus S} \widehat{f} = 0 \}$. Let
$P_Q: L^2(\R^d)\to L^2(Q)$ and $B_S:  L^2(\R^d)\to \cB(S)$
be the orthogonal projection operators defined by  
$P_Q[f] := \chi_Q f,  $ and $ B_S[f] := \mathcal{F}^{-1}[ \chi_S(\xi) \widehat{f}(\xi)],  ~  f\in L^2(\R^d)$,
%The operator $P_Q$ is a {\it cutoff operator} on the  space domain $Q$, 
%called a {\it $Q$-spacelimiting operator}, 
%and $B_S$  is a {\it cutoff operator} on the  frequency domain $S$,  
%called an {\it $S$-bandlimiting operator}. 
%The  operators $P_Q$ and $B_S$ do not commute,  i.e,  $P_QB_S\neq B_SP_Q$, however,  the products are  adjoint of each others, i.e.,  $(P_QB_S)^*  = B_SP_Q$.  
and define the associated {\it spatio-spectral limiting operator} by 
$\widetilde{\cT}_{Q,S}= B_SP_QB_S$. Note that $\widetilde{\cT}_{Q,S} = (P_QB_S)^* (P_Q B_S)$, and so
$\widetilde{\cT}_{Q,S}$ is positive semidefinite. Further,   $\widetilde{\cT}_{Q,S}$   is product of orthogonal projections, so it has norm $\leq 1$. 

\begin{comment}
Moreover,  $\widetilde{\cT}_{Q,S}: \cB(S)\to \cB(S)$  is an integral operator. Indeed, for any $S$-bandlimited function $f\in \cB(S)$,
 \begin{equation}\label{integral operator}
\widetilde{\cT}_{Q,S}[f](x) =   \int_{y\in \R^d}   K(x,y) f(y)   dy, ~   \quad x\in  \R^d
 \end{equation}
with  $K(x,y) =  h(x-y)$,  where $\hat h = \chi_S$, when $x,y\in Q$, and  $K(x,y)=0$ otherwise.

The kernel of $\widetilde{\cT}_{Q,S}$ is a Hilbert-Schmidt kernel, i.e., square integrable on $Q\times Q$, and 
\[
\|\widetilde{\cT}_{Q,S}\|_{HS}^2 = \int_{x\in Q} \int_{y\in Q} |K(x,y)|^2 dxdy.
\]
The operator  $\widetilde{\cT}_{Q,S}$ is   compact, and by the spectral theorem, the  eigenvalues of $\widetilde{\cT}_{Q,S}$  are between $0$ and $1$,  and they form an infinite monotonic sequence converging to zero: 
\begin{equation}
1 \geq  \lambda_0(Q,S) \geq \lambda_1(Q,S) \geq \cdots > 0 , \;\;  \text{with} \; \;  \lambda_k(Q,S) \rightarrow 0 \mbox{ as } k \rightarrow \infty. 
\label{eig_seq:defn}
\end{equation}
The corresponding eigenfunctions $\{ \Psi_{k} \}_{ k \geq 0}$ are in  $\mathcal{B}(S)$, the space of $S$-bandlimited functions in $L^2(\R^d)$.
\end{comment}

\begin{remark} \label{remk:SSLO_alt}

Returning to the variational problem \eqref{Raleigh-quo}, for any $\psi\in \cB(S)$ with $\|\psi\|_{L^2(\R^d)} =1$,
 \begin{align}\label{V1}
\cJ(\psi) = \langle \psi, P_Q \psi\rangle  = \langle B_S\psi, P_Q B_S \psi\rangle = \langle \psi, B_SP_Q B_S \psi\rangle =\langle \psi, \widetilde{\cT}_{Q,S}\psi\rangle.
\end{align} 
\begin{comment}
We also can write 
 \begin{align}\label{V2}
\cJ(\psi) = 
 \langle P_Q\psi, P_Q \psi\rangle  =   \langle P_QB_S\psi, P_Q B_S\psi\rangle  =  \| P_QB_S(\psi)\|^2 
\end{align}

Therefore,  the variational problem, and the question of which normalized bandlimited functions $\psi$ maximize the quantity $\cJ(\psi)$, is actually a question about the eigenvalues of SSLO operators. More precisely, our task is to maximize     
$$ 
\max_{  f\in \cB(S) , \|f\|=1} \| P_QB_S(f)\|^2 .
$$ 
The variational problem corresponding  to the operator $\cT_{Q,S} = P_QB_SP_Q$ has a different formulation  than the problem   \eqref{V1}-\eqref{V2}. This is simply due to the fact that the projection operators  $P_Q$ and $B_S$ do not commute. 
Although  the eigenfunctions of the operators $\widetilde{\cT}_{Q,S}$ and $\cT_{Q,S}$ are different,
\end{comment}
The eigenvalues of $\widetilde{\cT}_{Q,S}$ and $\cT_{Q,S} := P_Q B_S P_Q$ are identical. This is because $\widetilde{\cT}_{Q,S}= L_{Q,S}^*L_{Q,S}$ and $\cT_{Q,S} = L_{Q,S} L_{Q,S}^*$, with $L_{Q,S} = P_Q B_S$, while $L^*L$ has the same nonzero eigenvalues as $LL^*$  for any bounded operator $L$ on Hilbert space. Thus, 
 \begin{equation}
 \label{eig_seqn_2:defn}
 \{\lambda_k(Q,S)\}_{k \geq 0} \mbox{ is the sequence of positive eigenvalues of } \cT_{Q,S} = P_QB_SP_Q.
 \end{equation} 
Moving forward, we shall use \eqref{eig_seqn_2:defn} as the defining property of $\lambda_k(Q,S)$. 
 
% For further discussion on the relationship between  $\cT_{Q,S}$ and $\widetilde{\cT}_{Q,S}$, see, e.g.,  \cite{Karnik21}, Section 2. 
 \end{remark}

\section{Local sine bases with near-exponential frequency decay}
\label{sec:lsb}

In this section  we construct an orthonormal basis for $L^2(I)$, $I=[0,1]$, consisting of compactly supported functions whose Fourier transforms have rapidly decaying tails. For this, we define a variant of the \emph{local sine basis} of Coifman-Meyer \cite{CM}.
 
% We use $\omega$ to denote the frequency variable in $\R$, and let $\mathcal{F}(f)(\omega)  = \widehat{f}(\omega) = \int_\R e^{- i x \omega} f(x) dx$ be the Fourier transform of a function $f \in L^2(\R)$.

We define a family of intervals $\cW$, which forms a partition of $(0,1)$, given by
\begin{equation}
\label{W:defn}
\cW =  \{ [ 2^{-j-1}, 2^{-j} ) \}_{j \geq 1} \cup  \{ [ 1 - 2^{-j}, 1 - 2^{-j-1})\}_{j \geq 1}.
\end{equation}
Write $\delta_L \in (0,1]$ to denote the length of an interval $L \in \cW$. Note that, for $\delta > 0$,
\begin{equation}\label{good_geom1:eqn}
\delta_L \leq \delta \implies L \subset [0,2 \delta] \cup [ 1 - 2 \delta,1]. 
\end{equation}

The family $\cW$ is a realization of the \emph{Whitney decomposition} of $(0,1)$ -- see \cite{Fefferman-Israel} for a discussion of the Whitney decomposition of an open domain in $\R^d$.

Given $a>0$, we define the ``envelope  function'' $\Psi_a : \R \rightarrow \R$ by 
\begin{equation}\label{def:Psi}
\Psi_a(\xi) := \exp \left(- a \cdot \lvert \xi \rvert^{2/3} \right).
\end{equation}

The main result of the section follows: 

\begin{proposition}\label{basis:prop}
There exists a family of functions $\cB(I) = \{ \phi_{L,k} \}_{(L,k) \in \cW \times \N}$ in $L^2(I)$ such that 

\begin{itemize}
\item Each $\phi_{L,k}$  is a  $C^\infty$ function on the real line, with $\supp(\phi_{L,k}) \subset I$.
\item The family $\cB(I)$ is an orthonormal basis for $L^2(I)$.
\item Each $\phi_{L,k}$ satisfies the Fourier decay condition:  
\begin{equation}
\label{Fourier_decay_CM:eqn}
\lvert \widehat{\phi_{L,k}} (\xi) \rvert \leq A \delta_L^{1/2}  \sum_{\sigma \in \{\pm 1\}} \Psi_a\left( \delta_L  ( \xi - \sigma  \pi (k+1/2)/ \delta_L) \right) \quad \forall \xi\in \R.
\end{equation}

\end{itemize} 
In particular, one can take $A = 12$, $a=1/115$ in \eqref{Fourier_decay_CM:eqn}.
\end{proposition}

To prove Proposition \ref{basis:prop}, we let $\cB(I)$ be a variant of the local sine basis for $L^2(I)$. The classical construction of the local sine basis (see \cite{CM,Weiss93}) involves first selecting a special family of cutoff functions $b_L$, each supported on a neighborhood of an interval $L \in \cW$, and then multiplying each cutoff function by a family of sine waves (indexed by a frequency $k \in \N$) to obtain a double-indexed family $\phi_{L,k}$. Classically, it is known that the resulting family is an orthonormal basis for $L^2(I)$ satisfying the stated support conditions. To prove Proposition \ref{basis:prop}, we will show that if the cutoff functions are suitably regular -- specifically, if they belong to the Gevrey class $G^{3/2}$ --  then the decay condition \eqref{Fourier_decay_CM:eqn} will be satisfied. We recall the definition of the Gevrey class in Section \ref{Gevrey-class}.

Note, due to the rapid decay  of $\Psi_a(\xi)$ as $\xi \rightarrow \infty$, the  inequality \eqref{Fourier_decay_CM:eqn} can be interpreted to say that the  $L^2$-energy of $\widehat{\phi_{L,k}}(\xi)$ is sharply concentrated near two frequencies, where $\xi \sim \pm \pi (k+1/2)/ \delta_L$.

The next subsection is devoted to the proof of Proposition \ref{basis:prop}.

\subsection{Proof of  Proposition \ref{basis:prop}}

% We shall prove Proposition \ref{basis:prop} in this subsection.  The proof is divided into several steps.
 
\subsubsection{Construction of orthonormal  local sine  bases for $L^2(I)$}\label{construction:ONB}

 {\bf Step 1.} 
Let $\cW$ be the collection of subintervals of $I=[0,1]$ defined in \eqref{W:defn}. We enumerate the intervals of $\cW$ as follows: Write $\cW = \{L_j\}_{j \in \Z}$, for
\[
L_j = [\alpha_j, \alpha_{j+1}), \;\; \mbox{ where } \alpha_j = 2^{j-1} \mbox{ if }  j \leq 0, \mbox{ while } \alpha_j = 1-2^{-j-1} \mbox{ if } j \geq 1.
\]
That is, $L_{-2} = [1/8,1/4), L_{-1} = [1/4, 1/2), L_0 = [1/2,3/4), L_1 = [3/4, 7/8)$, etc.
 
Define sequences $\delta_j,\epsilon_j > 0$, $j \in \Z$, by 
\begin{align}
\label{eqn:delta_def}
&\delta_j = \alpha_{j+1}-\alpha_j = \mbox{ the sidelength of } L_j, 
\\
\label{eqn:eps_def}
&\epsilon_j = \delta_j/3 \mbox{ if } j \leq -1, \mbox{ and } \epsilon_j = 2 \delta_j/3 \mbox{ if } j \geq 0.
\end{align}
(See also \footnote{$\delta_j= 2^{-(j+2)}$ for $j\geq 0$, while $\delta_j= 2^{j-1}$  for $j\leq -1$.})

Note that, for all $j \in \Z$, 
we have 
\begin{equation} \label{eqn:eps_prop} \alpha_j + \epsilon_j = \alpha_{j+1} - \epsilon_{j+1}.
\end{equation} 
To prove  \eqref{eqn:eps_prop}, we shall show that $\epsilon_j + \epsilon_{j+1} = \delta_j$ for all $j$. First note the basic properties of the sequence $\delta_j$: We have $\delta_{-1} = 1/4$ and $\delta_0 = 1/4$, while $\delta_{j+1} = \delta_j/2$ for $j \geq 0$, and $\delta_{j} = \delta_{j+1}/2$ for $j \leq -2$. So, 
\begin{align*}
&j \leq -2 \implies \epsilon_j + \epsilon_{j+1} = \delta_j/3 + \delta_{j+1}/3 = \delta_j/3 +  2\delta_j/3 = \delta_j,\\
&j \geq 0 \implies \epsilon_j + \epsilon_{j+1} = 2\delta_j/3 + 2\delta_{j+1}/3 = 2\delta_j/3 + \delta_j/3 = \delta_j.
\end{align*}
 Finally, for $j =-1$, evaluate 
$$
\epsilon_{-1} + \epsilon_{0} = \delta_{-1}/3 + 2 \delta_{0}/3 = 1/12 + 1/6 = 1/4 = \delta_{-1}.$$
Thus, \eqref{eqn:eps_prop} is proven.

{\bf Step 2.}
Now we are ready to construct the family $\{\phi_{L,k}\}_{L \in \cW, k \in \N}$ in $L^2(I)$ satisfying the conditions in Proposition \ref{basis:prop}.

We follow the construction of the \emph{local sine basis} in \cite{CM,Weiss93}. To match the notation of \cite{Weiss93},
we set   $ b_{jk}: =\phi_{L_j,k}$ for $j\in \Z$. We shall define a family $\{ b_{jk} \}_{(j,k) \in \Z \times \N}$ in $L^2(I)$, satisfying the following conditions:

\begin{itemize}

\item Each $b_{jk}$ is a $C^\infty$ function with compact support, and $\supp(b_{jk}) \subset [ \alpha_j - \epsilon_j, \alpha_{j+1} + \epsilon_{j+1}]  \subset I$.

\item The family $\{ b_{jk} \}_{(j,k) \in \Z \times \N}$ is an orthonormal basis for $L^2(I)$.

\item  Each $b_{jk}$  has near-exponential frequency decay:  
\begin{equation}
\label{uniformbd}
\lvert \widehat{b_{jk}} (\xi) \rvert \leq A \delta_j^{1/2} \sum_{\sigma \in \{\pm 1\}} \Psi_a\left( \delta_j \cdot ( \xi -  \sigma \pi (k+1/2) / \delta_j ) \right).
\end{equation}
Here, $A, a  > 0$ are constants; we can take $ A = 12$, $a = 1/115$.
\end{itemize}
These conditions on $\{b_{jk}\}_{(j,k) \in \Z \times \N}$ imply  the conditions in the conclusion of Proposition \ref{basis:prop}.

We now give the construction of $\{b_{jk}\}$. We follow the presentation in \cite{Weiss93}, and fix an even non-negative function 
$\psi \in C_c^\infty(\R)$, with $\supp (\psi) \subset [-1,1]$ and $\int_{-\infty}^\infty \psi(t) dt = \pi/2$.  (A particular $\psi$ is defined in  \eqref{cutoff-function}.)  Let 
 $\theta(x) := \int_{-\infty}^x \psi(t) dt$. Then 
$\theta \in C^\infty(\R)$, and  
since $\psi$ is even, we have   $\theta(x) + \theta(-x) = \pi/2$. 
Also, $\theta(x) = 0$ for $x \leq -1$ and $\theta(x) = \pi/2$ for $x \geq 1$.

Define 
\[
s(x) := \sin(\theta(x)) \mbox{ and } c(x) := \cos(\theta(x)),
\]
so that  
$$c(x) = \cos(\pi/2 - \theta(-x)) = \sin(\theta(-x)) = s(-x).$$
Observe that $s$ in $C^\infty(\R)$ satisfies
\begin{align}\label{eqn:s1}
&s(x) = 0 \mbox{ for } x \leq -1, \; \; s(x) = 1 \mbox{ for } x \geq 1,\\
\label{eqn:s2}
& s^2(x) + s^2(-x) = 1  \mbox{ for all } x \in \R.
\end{align}
For $0 < \epsilon \leq 1$, define the dilated functions 
\begin{equation}\label{eqn:seps}
    s_\epsilon(x) = s(x/\epsilon), \;\;c_\epsilon(x) = c(x/\epsilon) = s_\epsilon(-x).
\end{equation}
Then $s_\epsilon \in C^\infty(\R)$ and
\[
\begin{aligned}
&s_\epsilon(x) = 0 \mbox{ for } x \leq -\epsilon, \; \; s_\epsilon(x) = 1 \mbox{ for } x \geq \epsilon,\\
& s_\epsilon^2(x) + s_\epsilon^2(-x) = 1  \mbox{ for all } x \in \R. 
\end{aligned}
\]
Following \cite{Weiss93}, for each interval $ [\alpha_j,\alpha_{j+1}]$, and for the prescribed values of $\epsilon_{j}$ and $\epsilon_{j+1}$, we define a ``bell function'' $b_j = b_{[\alpha_j,\alpha_{j+1}]} : \R \rightarrow \R$ by
\begin{equation}
    \label{eqn:bell1}
    b_j(x) = s_{\epsilon_j}(x-\alpha_j) c_{\epsilon_{j+1}}(x - \alpha_{j+1}) = s\left(\frac{x- \alpha_j}{\epsilon_{j}} \right)   s\left(\frac{\alpha_{j+1} - x}{\epsilon_{j+1}}\right), \;\; j \in \Z. 
\end{equation}
By applying \eqref{eqn:s1},
\begin{equation}
\label{LAB}
b_j \equiv 0 \; \mbox{on} \; \R \setminus \left[\alpha_j- \epsilon_j, \alpha_{j+1} + \epsilon_{j+1} \right].
\end{equation} 
It follows from \eqref{eqn:eps_prop} that $\alpha_j - \epsilon_j \geq \alpha_{j-1} + \epsilon_{j-1} > 0$, and similarly that $\alpha_{j+1} + \epsilon_{j+1} \leq \alpha_{j+2} - \epsilon_{j+2} < 1$. Thus, $\left[\alpha_j- \epsilon_j, \alpha_{j+1} + \epsilon_{j+1} \right] \subset [0,1]$ for all $j \in \Z$.

We define the ``local sine basis"  $\{ b_{jk}\}_{j \in \Z, k \in \N}$, where
\begin{equation}
    \label{eqn:lsb}
    b_{jk}(x) = b_{j}(x) \sqrt{ \frac{2}{\alpha_{j+1}-\alpha_j}} \sin \left(\pi\frac{2k+1}{2} \frac{x-\alpha_j}{ \alpha_{j+1}-\alpha_j} \right)  \qquad j\in \Z, k \in \N.
\end{equation}

Using \eqref{LAB} and the definition of $b_{jk}$ in \eqref{eqn:lsb}, we see that 
\begin{equation}
\label{supp_cond}
\supp (b_{jk}) = \supp(b_j) \subset [ \alpha_j - \epsilon_j, \alpha_{j+1} + \epsilon_{j+1}] \subset [0,1] \; \mbox{for all} \; j \in \Z, \; k \in \N,
\end{equation}
as stated in the first bullet point above. 
The following theorem establishes the second bullet. 

%We prove that the sequence of local sine functions defined in \eqref{eqn:lsb} is an orthonormal basis for $L^2(I)$: 
 
\begin{theorem}\label{thm:lcb}
The family $\{b_{jk}\}$ defined in \eqref{eqn:lsb} is an orthonormal basis for $L^2([0,1])$.
\end{theorem}

The papers \cite{CM,Weiss93} prove the orthonormality and completeness in $L^2(\R)$ for a basis of the form \eqref{eqn:lsb}, whenever $\alpha_j \in \R$, $j \in \Z$ is a strictly monotone sequence satisfying $\alpha_j \rightarrow + \infty$ as $j \rightarrow + \infty$ and $\alpha_j \rightarrow - \infty$ as $j \rightarrow - \infty$, while the sequence $\epsilon_j > 0$ (used in \eqref{eqn:bell1} to define the bell functions $b_j$) is assumed to satisfy \eqref{eqn:eps_prop}. By trivial modifications to the proof in \cite{CM,Weiss93}, one can see that the functions in \eqref{eqn:lsb} form an orthonormal basis for $L^2(I)$ ($I=[0,1]$) provided that $\alpha_j$ is a strictly monotone  sequence in $(0,1)$ satisfying $\alpha_j \rightarrow 1$ as $j \rightarrow \infty$, $\alpha_j \rightarrow -1$ as $j \rightarrow -\infty$, and  $\epsilon_j > 0$ satisfies \eqref{eqn:eps_prop}, where the bell functions $b_j$ are defined in \eqref{eqn:bell1}. In particular, the basis constructed here, where $\alpha_j$ are picked to be the endpoints of intervals of $\cW$, is an orthonormal basis for $L^2(I)$. We do not provide details for the proof of Theorem \ref{thm:lcb}, but instead, refer the reader to \cite{CM,Weiss93}, where the  argument is given.

Thus, the second bullet point is satisfied for the family $\{b_{jk}\}$.

%\vskip.12in 
% Azita: I remove the vskips as above, since they lead to inconsistent spacing in the document -- that is, some paragraphs are separated by different amounts from other paragraphs. If you'd like to change the paragraph spacing, can I recommend implementing that change at the top of the document when the style class is defined?

We point out that the construction of the local sine basis is somewhat flexible, due to the choice of the smooth cutoff function $\psi \in C^\infty_c$, which determines $\theta$ and $s$ in $C^\infty$. Recall, from \eqref{eqn:s1}, \eqref{eqn:s2} that $s$ resembles a smoothed version of the Heavyside function, and it is used to define the bell functions $b_j$ in \eqref{eqn:bell1}.

As we shall see, if $\psi$ is picked carefully, then $s$ will belong to  the Gevrey class $G^{3/2}$, and it will follow that the basis $\{b_{jk}\}$ satisfies \eqref{uniformbd}, completing the proof of Proposition \ref{basis:prop}. 

In the next subsection we give the definition of the Gevrey class, and state some of its basic properties. In the following subsection we define the cutoff function $\psi\in C_c^\infty$ and prove the required properties of $s$. In the final subsection, we prove the Fourier decay condition \eqref{uniformbd} for the basis $\{b_{jk}\}$.

\subsubsection{Gevrey class}\label{Gevrey-class}
\label{sec:Gevrey}

We denote   by $C^\infty(I)$  the class of all $C^\infty$ function $f : I \rightarrow \R$ defined on an interval $I$. We denote the $k$-fold derivative of $f$ by $f^{(k)}$ and write $\| f \|_\infty$ for the  $L^\infty$ norm of $f$ on $I$.

Following H\"ormander's book \cite{Hor2} (see also  \cite{Rod1}), we say that a  function  $f\in C^\infty(I)$ is in the Gevrey class $G^r$ ($r \geq 1$) provided that for all compact sets $K \subset I$ there exist $C_K, R_K > 0$ such that $|f^{(k)}(x)| \leq C_K R_K^{k} (k!)^r$ for all $x \in K$ and $k \geq 0$. Observe that $G^r\subset C^\infty(I)$.

We say that $f$ is in the Gevrey class $G^r_0$ provided that there exist constants $C, R > 0$    such that (1) $|f^{(k)}(x)| \leq C R^{k} (k!)^r$ for all $x \in I$ and $k \geq 0$. Sometimes we shall want to record the constants in the Gevrey condition: We say that $f$ is in the class $G_0^r(C,R)$ when (1) is valid for given $C$ and $R$.
Note that $G^r_0 = \bigcup_{C,R > 0} G^r_0(C,R)$,  
 and $G_0^{r} \subset L^\infty(I) \cap C^\infty(I)$.

 Gevrey classes interpolate between the space $C^\infty$ (identified with $G^r$ in the limit $r \rightarrow \infty$) and the space of analytic functions $\mathcal{A}$ (identified with $G^1$).

Note that the functions $\sin(x)$ and $\cos(x)$ belong to the Gevrey class $G^r_0(1,1)$ for any $r \geq 1$.
Also, for each $r > 1$, the Gevrey class $G^r$ contains a nonzero compactly supported $C^\infty$ function (see \cite{Hor2}).

We now state two standard results on the Gevrey class, which will be needed in what follows. For completeness, we provide proofs in the appendix (Section \ref{appendix:sec}).
The first result states that the Gevrey class is closed under multiplication and composition.

\begin{proposition}
\label{prop:Gevrey1} For any $r > 1$, the Gevrey class $G_0^r$ is an Algebra closed under multiplication and composition. In particular, 
\begin{enumerate}
    \item Suppose $f : I \rightarrow \R$ is in  $G^r_0(C,R)$ and $g : I \rightarrow \R$ is in $G^r_0(D,S)$. Then: $f g \in G_0^r(CD,2 \max \{R,S\})$.
    \item Suppose $f : I \rightarrow \R$ is in  $G^r_0(C,R)$ and $g : J \rightarrow \R$ is in $G^r_0(D,S)$, and suppose $f(I) \subset J$. Then: $g \circ f : I \rightarrow \R$ is in $G_0^r(D,T)$, with $T = R(1+CS)$.
\end{enumerate}
\end{proposition}

By the Paley-Wiener theorem, if $f \in C_c^\infty(\R)$ then $\widehat{f}$ is analytic, and its restriction on the real line $\widehat{f}(\xi)$
has polynomial (Schwartz) decay as $\xi \rightarrow \infty$, namely, for every $N \geq 1$ there exists $C_N > 0$ such that $| \widehat{f}(\xi)| \leq C_N(1+|\xi|)^{-N}$ for all $\xi\in \R$. One can improve the polynomial decay to near-exponential decay provided $f$ is in the Gevrey class.  This is the next result. Also see Lemma 12.7.4 of \cite{Hor2}, and \cite{DH} for related results.

\begin{proposition}\label{prop:Gevrey_PW}
Suppose that $B : \R \rightarrow \R$ is $C^\infty$, satisfying:
\begin{enumerate}[(a)]
\item $B$ is supported on $[-2,2]$.
\item $B \in G^{3/2}_0(C,R)$ for some $C,R > 0$. 
\end{enumerate}
Set $A=12C$ and $a=1/(2eR^{2/3})$. 
Then $\lvert  \widehat{B} \lvert$ has near-exponential decay: 
 $$  \lvert \widehat{B} (\xi) \rvert \leq A  \exp( - a \cdot | \xi|^{2/3})   ~~  \text{for all}~  ~  \xi \in \R.$$
\end{proposition}

\subsubsection{Construction of cutoff functions and the modified Heavyside function $s$}\label{construction:cutoff}

In this section we define the function  $\psi \in C^\infty_c$ used in the construction of the local sine basis $\{ b_{jk}\}$.

Define the bump function 
\begin{equation} \label{c0}
v(x) = \left\{
        \begin{array}{ll}
            e^{- (1-x^2)^{-2}}  & \quad x \in (-1,1)\\
            0   & \quad x \in (-\infty, - 1] \cup [1,\infty).
        \end{array}
    \right.
 \end{equation}
 
\begin{comment} 
Python code: 

import numpy as np
import matplotlib.pyplot as plt

def v(x):
    if -1 < x < 1:
        return np.exp(-((1 - x**2)**-2))
    else:
        return 0

x_values = np.linspace(-2, 2, 1000)  # Generate x-values from -2 to 2
y_values = [v(x) for x in x_values]  # Calculate corresponding y-values

plt.plot(x_values, y_values)
plt.xlabel('x')
plt.ylabel('v(x)')
plt.title('Graph of v(x)')
plt.grid(True)
plt.show()

\end{comment}

\begin{lemma}\label{der_bounds1} 
The function $v$ given in \eqref{c0} is in $C^\infty(\R)$, with $\supp(v) \subset [-1,1]$. Further, $v$ is in the Gevrey class $G^{3/2}_0(1,C_1)$ on $\R$, for $C_1 = 2 + 3 \sqrt{2}$.
Thus, 
$v$  
satisfies
\[ \| v^{(k)} \|_{\infty} \leq C_1^{k} \cdot (k!)^{3/2} \;\;\; \mbox{for all} \; k \geq 0.
\]
\end{lemma}

\begin{proof}  We shall demonstrate that  $v$ is the composition of two functions in the Gevery class. Then we use Proposition \ref{prop:Gevrey1} to complete the proof.

Note that $v(t) = g(1-t^2)$ for $t \in (-1,1)$, where $g : (0,\infty) \rightarrow \R$ is the function $g(x)  = e^{- x^{-2}}$. 
We claim that that $g$ is in the Gevrey class $G^{3/2}_0(1,3/\sqrt{2})$ on $(0,\infty)$. Indeed, apply the Cauchy integral formula for $g(z)= e^{-z^{-2}}$, with $\gamma$ a circle in $\C$ of center $x$ and radius $x/2$,
\[
g^{(k)}(x) = \frac{k!}{2 \pi i} \oint_\gamma \frac{e^{-z^{-2}}}{(z-x)^{k+1}} dz.
\]
On $\gamma$ the function $\mathrm{Re}(z^{-2})$ achieves its minimum value at $z = 3x/2$, hence, $\sup_{z \in \gamma} | e^{-z^{-2}}| = e^{-4/9x^2}$; thus,
\[
|g^{(k)}(x)| \leq k! \left(2/x \right)^k e^{-4/{9x^2}} = k! 3^k \left(4/9x^2 \right)^{k/2} e^{-4/{9x^2}}.
\]
By elementary calculus, $y^R e^{-y} \leq (R/e)^R$ for $y,R > 0$. Use this inequality with $y = 4/{9x^2}$ and $R = k/2$,  to get
\[
|g^{(k)}(x)| \leq  k! 3^k (k/2e)^{k/2}.
\]
By Stirling's formula, $\sqrt{2 \pi k} (k/e)^k \leq k!$ (see \cite{Robbins}), so that
\[
|g^{(k)}(x)| \leq  (3/\sqrt{2})^k (k!)^{3/2}.
\]
Therefore, $g$ is in the Gevrey class $G^{3/2}_0(1,3/\sqrt{2})$ on $(0,\infty)$.

On the other hand, the function $f(t) = 1-t^2$ is in the Gevrey class $G^{3/2}_0(1,2)$ on $I = (-1,1)$, since $|f(t)| \leq 1$, $|f'(t)| \leq 2$, and $|f''(t)| \leq 2$ on $I$, while the higher derivatives vanish. By the composition law for Gevrey functions (part 2 of Proposition \ref{prop:Gevrey1}) and $f(I)\subset (0,\infty)$, $v(t) = e^{- (1-t^2)^{-2}}$ is in the Gevrey class $G^{3/2}_0(1,2(1+3/\sqrt{2}))$ on $(-1,1)$. Because $v \equiv 0$ on $\R \setminus (-1,1)$, we deduce that $v$ is in $G^{3/2}_0(1,2(1+3/\sqrt{2}))$ on $\R$.
This completes the proof of the lemma.

\end{proof}

Next, we  define the cutoff  function $\psi \in C_c^\infty(\R)$ using the $v$ in \eqref{c0}. Set  
\begin{align}\label{cutoff-function}
\psi(y) := (\pi/2)\left( \int_\R v  dx \right)^{-1} v(y).
\end{align}

Clearly, $\psi$ is even (since $v$ is even), $\psi \geq 0$, $\supp (\psi) \subset [-1,1]$ and $\int \psi dy = \pi/2$. As in the discussion preceding \eqref{eqn:s1}, we let 

\begin{equation}
    \label{s_theta_defn}
\theta(x) := \int_{-\infty}^x \psi(y) dy, \quad \text{and} \quad  s(x) := \sin(\theta(x)).
\end{equation}
This $s$ satisfies \eqref{eqn:s1}, \eqref{eqn:s2}. Next, we show that $s$ is in the Gevrey class $G^{3/2}_0$.

\begin{lemma}\label{lem:cutoff}
The function $s$ is in $G^{3/2}_0(1,C_2)$ on $\R$, for $C_2 = (2+3 \sqrt{2})(1+\pi/2)$.
\end{lemma}
\begin{proof}
From \eqref{c0}, by numerical integration,
\[
\int_\R v dx = \int_{-1}^1 e^{-(1-x^2)^{-2}} dx = 0.3402 \cdots \geq 1/3.
\]
So, according to the definition of $\psi$ above, $\psi = c v$ for $|c| \leq 3 \pi/2$. If $f$ is in the Gevrey class $G^{r}_0(A,B)$ then $cf$ is in $G^{r}_0(|c|A,B)$ for any constant $c$. In Lemma \ref{der_bounds1} we showed that $v \in G^{3/2}_0(1,C_1)$, for $C_1 = 2 + 3 \sqrt{2}$. 
Thus, $\psi \in G^{3/2}_0(3\pi/2,C_1)$. Given that $\theta(x) = \int_{-\infty}^x \psi(y) dy$, we have 
\[
\| \theta \|_{\infty} \leq \int \psi dy = \pi/2.
\]
Using that $\theta' = \psi \in G^{3/2}_0(3\pi/2 ,C_1)$, we obtain 
\begin{equation*}
\| \theta^{(k)} \|_\infty =  \| \psi^{(k-1)} \|_\infty \leq (3 \pi/2)  C_1^{k-1} (k!)^{3/2} = (3 \pi/2 C_1) C_1^k (k!)^{3/2} \qquad \mbox{for} \; k \geq 1.
\end{equation*}
Note that $(3 \pi/2 C_1) \leq \pi/2$, since $C_1 \geq 5$, so the previous two lines establish that $\theta$ is in $G_0^{3/2}(\pi/2,C_1)$.    

Note that $\sin(x)$ is in $G^{3/2}_0(1,1)$. From part 2 of Proposition \ref{prop:Gevrey1}, $s(x) = \sin(\theta(x))$ is in $G^{3/2}_0(1,C_2)$, for $C_2 = C_1(1+\pi/2) = (2 + 3 \sqrt{2})(1+\pi/2)$, as claimed.

\end{proof}

\subsubsection{Near-exponential frequency decay of the local-sine basis}

Here, we show that the local sine basis $\{b_{jk}\}_{(j,k)\in \Z \times \N}$ in \eqref{eqn:lsb} satisfies the Fourier decay condition \eqref{uniformbd}. This concludes the verification of the properties of $b_{jk}$, and completes the proof of Proposition \ref{basis:prop}.

For $j \in \Z$, let $B_j(x) := b_j( x \cdot \delta_j + \alpha_j) =  s \left(x \cdot \frac{\delta_j}{\epsilon_j} \right) \cdot s \left( (1-x) \cdot \frac{\delta_j}{\epsilon_{j+1}} \right)$ (see \eqref{eqn:bell1}). By definition of $\epsilon_j$ in \eqref{eqn:eps_def}, $\delta_j/\epsilon_j \leq 3$ and $\delta_j/\epsilon_{j+1} \leq 3$ for all $j$. Therefore, by Lemma \ref{lem:cutoff}, $\| s^{(k)} \|_{L^\infty} \leq C_2^k (k!)^{3/2}$ for $C_2 = (2 + 3 \sqrt{2})(1+\pi/2)$, and by the chain and product rules for differentiation we obtain 
\[
\| B_j^{(k)} \|_{L^\infty} \leq 3^k \cdot 2^k \cdot C_2^k \cdot (k!)^{3/2} = C_3^k (k!)^{3/2} \quad \mbox{ for } k \geq 0,
\]
where $C_3 = 6 C_2 = (12 + 18 \sqrt{2}) (1+\pi/2) $. Thus, $B_j$ is in the Gevrey class $G^{3/2}_0(1,C_3)$. Note that $b_j$ is supported on $[\alpha_j - \epsilon_j, \alpha_{j+1} + \epsilon_{j+1}] \subset [\alpha_j-2\delta_j, \alpha_{j} + 2 \delta_j]$ -- here, we apply \eqref{supp_cond} and note that $\alpha_{j+1} = \alpha_j + \delta_j$, while $\epsilon_j \leq 2 \delta_j/3$ (by \eqref{eqn:eps_def}) and $\epsilon_{j+1} \leq 2 \delta_j/3$ (by the discussion after \eqref{eqn:eps_prop}). Thus, $B_j$ is supported on $[-2,2]$. By applying Proposition \ref{prop:Gevrey_PW}, with $C=1$ and $R = C_3$, we deduce that
\[
\lvert \widehat{B_j}(\xi) \rvert \leq A   \exp \left( - a   | \xi |^{2/3} \right)  \;\; \mbox{for} \; \xi \in \R,
\]
where $A = 12$ and $a = \left[ 2e ((12 + 18 \sqrt{2}) (1+\pi/2) )^{2/3} \right]^{-1} \geq \frac{1}{115}$.

Since $b_j(x) = B_j\left(\frac{x - \alpha_j}{\delta_j}\right)$, by scaling properties of the Fourier transform we conclude that
\begin{equation}
\label{expdecay}
\lvert \widehat{b_j}(\xi) \rvert \leq A  \delta_j \exp \left( - a | \delta_j \xi |^{2/3} \right) \;\; \mbox{for} \; \xi \in \R. 
\end{equation}
Using $\sin ( \gamma) = \frac{1}{2i} \left(e^{i \gamma} - e^{-i \gamma}\right)$ in equation \eqref{eqn:lsb}, and recalling $\delta_j = \alpha_{j+1} - \alpha_j$ (see \eqref{eqn:delta_def}), and the identity $\mathcal{F}(e^{i (x-x_0) \xi_0}f(x)
 )(\xi) = \widehat{f}(\xi - \xi_0) e^{-ix_0\xi_0}$ for the composition of the Fourier transform and a modulation, we obtain   
\begin{align*}
\widehat{b_{jk}}(\xi) = \frac{\sqrt{2}}{2i \sqrt{\delta_j}} \Biggl[ & \widehat{b_j}\left(\xi - \pi (k + 1/2) \delta_j^{-1} \right) \cdot e^{ - i \pi (k + 1/2) \alpha_j/\delta_j } \\
- & \widehat{b_j} \left(\xi + \pi (k + 1/2) \delta_j^{-1} \right) \cdot e^{ i \pi (k + 1/2)  \alpha_j/\delta_j }\Biggr].
\end{align*}
Then by \eqref{expdecay} and the triangle inequality,
\begin{equation}
\label{uniformbd1_a}
\lvert \widehat{b_{jk}} (\xi) \rvert \leq A  \delta_j^{\frac{1}{2}} \sum_{\sigma = \pm 1 } \exp \left( - a     \left\lvert  \delta_j   \xi - \sigma  \pi \left( k + 1/2 \right) \right\rvert^{2/3} \right).
\end{equation}
This completes the proof of \eqref{uniformbd}. Thus, we have completed the proof of  Proposition \ref{basis:prop}.

\section{Technical tools}

\subsection{Estimating the eigenvalues of positive semidefinite compact operators} 

%Given a real Hilbert space $\mathcal H$, and a  compact, positive semidefinite operator $T:\mathcal H\to \mathcal H$, let $\{\lambda_j(T)\}_{j\geq 1}$ denote the sequence of positive eigenvalues of $T$, counted with multiplicity, and sorted in non-increasing order.   
%Let $\{\psi_k\}_{k\in \mathcal I}$  be an orthonormal basis for $\mathcal H$. 
%In the following  lemma we prove that, given a partition of  the orthonormal basis $\{\psi_k\}_{k\in \mathcal I}$ satisfying \eqref{func_lemma:eqn} for some $\epsilon\in (0,1/2)$, one can estimate the number of the eigenvalues of $T$ in the ``transition" region   $(\epsilon, 1-\epsilon)$, and the number of eigenvalues in $(\epsilon,1]$. 

The following lemma is based on a variant of Lemma 1 of \cite{Israel15}, and will allow us to estimate the distribution of the eigenvalues of a positive compact operator, assuming the operator satisfies the Hilbert-Schmidt type condition \eqref{func_lemma:eqn} with respect to some basis. This result will be used in the proof of Theorem \ref{mainthm:cube_convex}.

\begin{lemma}\label{func_anal:lem}
Let $\mathcal{H}$ be a real Hilbert space. Given a positive semidefinite compact operator $T : \mathcal{H} \rightarrow \mathcal{H}$, let $\lambda_j(T)$, $j \geq 1$, be the eigenvalues of $T$, counted with multiplicity, and sorted in non-increasing order. 

Suppose that there exists an orthonormal basis $\{\psi_k\}_{k\in \cI}$ for $\mathcal H$, and suppose the index set $\cI$ can be partitioned as $\cI = \cI_{low} \cup \cI_{res} \cup \cI_{hi}$, with $\cI_{res}$ and $\cI_{low}$ finite sets, such that, for some $\epsilon \in (0,1/2)$, 
\begin{equation}\label{func_lemma:eqn}
\sum_{k\in \cI_{hi}} \|T \psi_k\|^2 + \sum_{k\in \cI_{low}} \|(I-T)\psi_k\|^2\leq \epsilon^2.
\end{equation}

Given $\epsilon \in (0,1/2)$, let $M_\epsilon(T) := \# \{ j : \lambda_j(T) > \epsilon \}$ and $N_\epsilon(T) := \# \{ j : \lambda_j(T) \in (\epsilon, 1 - \epsilon) \}$.
Then 
$$
N_\epsilon(T) \leq  \#(\cI_{res}), \quad \text{and} ~~ | M_\epsilon(T) - \#(\cI_{low})  | \leq \# ( \cI_{res}). 
$$
\end{lemma}

\begin{proof}
Let $K_1 = \#(\cI_{low})$ and $K_2 = \#(\cI_{low}) + \#(\cI_{res})$. We will prove that
\begin{equation}\label{eigen_bd:eqn}
\lambda_{K_1}(T) \geq 1 - \epsilon > \epsilon, \;\; \mbox{and } \lambda_{K_2+1}(T) \leq \epsilon.
\end{equation}
From these bounds, it follows that $\{ j : \lambda_j(T) \in (\epsilon, 1 - \epsilon) \} \subset \{ K_1 + 1,\cdots, K_2 \}$, and thus, $N_\epsilon(T) \leq K_2 - K_1$ and $K_1 \leq M_\epsilon(T) \leq K_2$. So, the desired bounds on $M_\epsilon(T)$ and $N_\epsilon(T)$ follow from \eqref{eigen_bd:eqn}.

Condition \eqref{func_lemma:eqn} is equivalent to the statement that
\[
\| T - T P_{res} - P_{low} \|_{HS} \leq \epsilon,
\]
where $P_{low} : \cH \rightarrow X_{low}$ and $P_{res} : \cH \rightarrow X_{res}$ are  orthogonal projection operators for the orthogonal subspaces $X_{low} := \mathrm{span} \{ \psi_k \}_{k \in \cI_{low}}$ and $X_{res} := \mathrm{span} \{ \psi_k \}_{k \in \cI_{res}}$ in $\cH$, and where $\| \cdot \|_{HS}$ denotes the Hilbert-Schmidt norm. Recall that the operator norm is bounded by the Hilbert-Schmidt norm, i.e. $\| T_0 \| \leq \| T_0 \|_{HS}$ for any linear operator $T_0 : \cH \rightarrow \cH$.  Thus, condition \eqref{func_lemma:eqn} implies
\begin{equation}
    \label{func_lemma2:eqn}
    T = T P_{res} +  P_{low} + T_{err}, \mbox{ with } \| T_{err} \| \leq \epsilon.
\end{equation}

We apply the variational characterization of eigenvalues (the Courant-Fischer-Weyl lemma) to prove \eqref{eigen_bd:eqn}. Recall that for any positive semidefinite (self-adjoint) compact  operator $T$ on $\cH$,
\begin{equation}\label{MinMaxChar}
\lambda_j(T) = \sup_{ \dim (W) = j} \min_{x \in W, \; \|x \| = 1}  \langle Tx,x \rangle,
\end{equation}
where the supremum is over all $j$-dimensional subspaces $W \subset \cH$.

We first estimate $\lambda_{K_1}(T)$. Note that $\dim (X_{low}) = \# \cI_{low} = K_1$. For $x \in X_{low}$ with $\| x \| = 1$, we have $P_{res}x = 0$ and $P_{low}x = x$, thus, from \eqref{func_lemma2:eqn},
\[
\langle Tx, x \rangle = \langle x,x \rangle + \langle T_{err} x, x \rangle \geq 1 - \epsilon.
\]
By \eqref{MinMaxChar}, this implies $\lambda_{K_1}(T) \geq 1-\epsilon$.

We next estimate $\lambda_{K_2+1}(T)$. Let $W$ be an arbitrary subspace of $\cH$ with $\dim (W) = K_2 + 1$. Let $X_{hi}$ be the closure of $\mathrm{span} \{ \psi_k  \}_{k \in \mathcal{I}_{hi}}$. The orthogonal decomposition $\cH = X_{low} \oplus X_{res} \oplus X_{hi}$ implies that $\mathrm{codim}(X_{hi}) = \dim(X_{low}) + \dim(X_{res}) =  \# \mathcal{I}_{low} + \# \mathcal{I}_{res} = K_2$. Since $\mathrm{codim}(X_{hi}) < \dim(W)$, it holds that $X_{hi} \cap W \neq \{0\}$. Fix $y \in X_{hi} \cap W$ with $\| y \| = 1$. Because $y \in X_{hi}$, also $y$ belongs to the kernel of $T P_{res} + P_{low}$. Thus, from \eqref{func_lemma2:eqn}, $\langle Ty,y \rangle  = \langle T_{err} y, y \rangle \leq \epsilon$. Since $W$ is an arbitrary $(K_2+1)$-dimensional subspace of $\cH$, we have, by \eqref{MinMaxChar}, that $\lambda_{K_2+1}(T) \leq \epsilon$.

This concludes the proof of \eqref{eigen_bd:eqn}, and with it, the proof of the lemma.

\end{proof}

\subsection{Counting lattice points in a convex body}\label{sec:lattice}

The main result of this section is Corollary \ref{cor:lat}, which will be used in the proofs of Lemma \ref{count1:lem} and Lemma \ref{count2:lem} in Section \ref{wave-packet}.

To start, in Lemma \ref{steiner:lem} below, we recall the Steiner formula \eqref{steiner1:eqn} for the volume of the outer neighborhood of a convex body $K$, and derive a related inequality \eqref{steiner2:eqn} for the volume of the inner neighborhood of $K$. 
%These geometric inequalities will be used in the proofs of Lemma \ref{lem:lat} and Corollary \ref{cor:lat}.

By a \emph{convex body} in $\R^d$, we  mean a compact and convex subset of $\R^d$. A convex body $K$ is said to be \emph{symmetric} provided that $x \in K \implies -x \in K$. We write $B(r)$ for the closed ball in $\R^d$ centered at $0$ with radius $r$.

\begin{lemma}\label{steiner:lem}
Let $K$ be a symmetric convex body in $\R^d$.
Then $r \mapsto \mu_d(K \oplus B(r))$ is a polynomial function of degree at most $d$, with coefficients determined by $K$ and $d$. Specifically,
\begin{equation}\label{steiner1:eqn}
\mu_d(K \oplus B(r)) = \mu_d(K) + \sum_{j=0}^{d-1} \kappa_{d-j}  V_{j}(K)  r^{d-j}, \qquad \mbox{for all } r > 0,
\end{equation}
where $\kappa_\ell$ is the $\ell$-dimensional Lebesgue measure of the unit ball in $\R^\ell$, and $V_{j}(K) \geq 0$ is a coefficient, called the \underline{$j$'th intrinsic volume} of $K$. 

Further:
\begin{equation}\label{steiner2:eqn}
\mu_d(K \ominus B(r)) \geq \mu_d(K) - \sum_{j=0}^{d-1} \kappa_{d-j} V_j(K) r^{d-j}  ,  \qquad \mbox{for all } r > 0.
\end{equation}
\end{lemma}
\begin{proof}
Equation \eqref{steiner1:eqn} is the classical Steiner formula -- see \cite{Sch1} for details.

Now, \eqref{steiner2:eqn} follows from \eqref{steiner1:eqn} and the following inequality, which is valid for any convex body $K$ in $\R^d$ and $r > 0$,
\begin{equation}\label{inner_vs_outer:eqn}
    \mu_d( K \setminus (K \ominus B(r))) \leq \mu_d((K \oplus B(r)) \setminus K).
\end{equation}
By density of polytopes in the set of convex bodies, it suffices to prove \eqref{inner_vs_outer:eqn} for a polytope $K$. The region $K \setminus (K \ominus B(r))$ is the set of all $x \in K$ with  $d(x,\partial K) < r$, and $(K \oplus B(r)) \setminus K$ is the set of all $x \in \R^d \setminus K$ with $d(x,\partial K) \leq r$. We can write $\partial K$ as a union of facets $F$ of dimension $d-1$. Each facet carries a unit normal $n_F \in \R^d$, $|n_F| = 1$, directed away from $K$ (i.e., $n_F$ is an outward normal to $\partial K$). Then
\[
K \setminus (K \ominus B(r)) \subset \bigcup_F F_r^-,
\]
where the union is over all facets $F$, and $F_r^- := \{ x - t n_F : x \in F, 0 \leq t \leq r \}$ is a prism in $\R^d$ with base $F$ and thickness $r$ directed toward $K$. Letting $F_r^+ := \{ x + t n_F : x \in F, 0 \leq t \leq r \}$ denote the corresponding prism directed away from $K$, we have
\[
(K \oplus B(r)) \setminus K \supset \bigcup_F F_r^+.
\]
By convexity of $K$ the family $\{F_r^+ : F \mbox{ facet of } K\}$ is disjoint. Because $\mu_d(F_r^-) = \mu_d(F_r^+)$ for all facets $F$, the inequality \eqref{inner_vs_outer:eqn} follows by subadditivity of Lebesgue measure.

\begin{comment}We now apply \eqref{steiner1:eqn} to prove \eqref{steiner2:eqn}. Write $r_{1}(K) \in [0.\infty)$ to denote the inner radius of $K$, defined by
\[
r_1(K) = \sup \{ r \geq 0 : B(r) \subset K \}.
\]
By compactness of $K$, the supremum is achieved, thus, $B(r_1(K)) \subset K$. For $\rho \leq r_{1}(K)$, one has $0 \in K \ominus B(\rho)$ -- in particular, $K \ominus B(\rho) \neq \emptyset$. Further, one can check that $(K \ominus B(0,\rho) ) \oplus B(0, \rho) = K$ for $\rho \leq r_1(K)$. We apply Steiner's formula \eqref{steiner1:eqn} with $K \ominus B(\rho)$ in place of $K$, to give
\[
\mu_d(K) = \mu_d(K \ominus B(\rho)) + \sum_{j=0}^{d-1}  \gamma_{d-j} V_j(K \ominus B(\rho)) \rho^{d-j} \qquad (\rho \leq r_{1}(K)).
\]
Intrinsic volumes are inclusion monotone, i.e., $V_j(K_1) \leq V_j(K_2)$ for $K_1 \subset K_2$ -- this is because intrinsic volumes can be identified as mixed volumes, and mixed volumes are inclusion monotone (see \cite{Sch1}). Hence, $V_j(K \ominus B(\rho)) \leq V_j(K)$. Thus,
\[
\mu_d(K \ominus B(\rho)) \geq \mu_d(K) - \sum_{j=0}^{d-1}  \gamma_{d-j} V_j(K) \rho^{d-j}  \qquad (\rho \leq r_{1}(K)).
\]
If $\rho = r_{1}(K)$ then $K \ominus B(\rho)$ is convex with an empty interior, and therefore $\mu_d(K \ominus B(\rho))=0$; thus, the RHS of the above inequality is $\leq 0$ for $\rho = r_{1}(K)$. Also, the RHS is a decreasing function of $\rho$, and the LHS is $0$ if $\rho > r_{1}(K)$ since then $K \ominus B(\rho) = \emptyset$. Therefore, the inequality persists for $\rho > r_{1}(K)$, and this proves \eqref{steiner2:eqn}.
\end{comment}

\end{proof}

%\begin{remark}
%Another proof of \eqref{steiner2:eqn} starts by noting that \eqref{steiner1:eqn} implies $\mu_d((K \oplus B(r)) \setminus K) = \sum_{j=0}^{d-1} \kappa_{d-j} V_j(K) r^{d-j}$. It is geometrically obvious that $\mu_d( K \setminus (K \ominus B(r))) \leq \mu_d( (K \oplus B(r)) \setminus K)$, since the ``outer half'' of the $r$-neighborhood of $\partial K$ has larger volume than the ``inner half'' of the $r$-neighborhood of $\partial K$ whenever $K$ is convex. This inequality and the previous identity together imply \eqref{steiner2:eqn}. We could not find a reference in the literature for the ``obvious'' inequality, so, for the sake of rigor, we included the slightly longer proof of \eqref{steiner2:eqn} just above.
%\end{remark}

For context, we mention an alternate characterization for the intrinsic volumes $V_j(K)$. The leading term $\mu_d(K)$ in \eqref{steiner1:eqn} is the $d$'th intrinsic volume $V_d(K)$. Further, $V_0(K) = 1$, and $V_{d-1}(K) = \frac{1}{2} \cH_{d-1}(\partial K)$ is half the $(d-1)$-dimensional Hausdorff measure of the boundary of $K$. In general, $V_j(K)$  is a constant multiple of the average diameter of a (random) projection of the body $K$ onto a $j$-dimensional subspace. See \cite{Sch1} for details. 
%We do not use these remarks in the remainder of the paper, they are meant only to describe the geometric interpretations of the coefficients in the Steiner formula.

One can upper bound the intrinsic volumes of a symmetric convex body $K$ via the outer radii of $K$. For $1 \leq j \leq d$, the \emph{$j$'th outer radius} of $K$, written $R_j(K)$, is the smallest $R$ such that $K \subset ( B(R) \cap W) \oplus W^\perp$ for some $j$-dimensional subspace $W \subset \R^d$ -- here, $W^\perp$ is the orthogonal complement of $W$. Equivalently: $R_j(K)$ is the smallest $R$ such that $K$ is contained in a cylindrical tube with cross section given by a $j$-dimensional ball of radius $R$. Observe, 
\[
R_1(K) \leq R_2(K) \leq \cdots \leq R_d(K) = \diam(K)/2.
\]
To illustrate the definition: Let $E = E(r_1,\cdots,r_d)$ be the ellipsoid in $\R^d$ given by all points $x \in \R^d$ such that $\sum_j x_j^2/r_j^2 \leq 1$. Then $R_1(E),\cdots, R_d(E)$ is the non-decreasing rearrangement of $r_1,\cdots,r_d$.

In Theorem 1.2 of \cite{Henk08}, the following bound is given:
\[
V_j(K) \leq 2^j s_j(R_1(K),\cdots,R_d(K)) \qquad \mbox{ for } j=0,\cdots,d,
\]
where $s_j(x_1,\cdots,x_d)$ is the $j$'th elementary symmetric polynomial, 
\[
s_j(x_1,\cdots,x_d) := \sum_{i_1 < i_2 < \cdots < i_j} \prod_{k=1}^j x_{i_k},
\]
and we set $s_0(x_1,...,x_d) = 1$. Since $R_1(K) \leq \cdots \leq R_d(K)$, we conclude
\begin{equation}\label{intr_vol_bd:eqn}
\begin{aligned}
&V_j(K) \leq 2^j \binom{d}{j} \prod_{m=d-j+1}^d R_m(K) =: C_j(K), \qquad 1 \leq j \leq d, \\
&V_0(K) = 1 =: C_0(K)
\end{aligned}
\end{equation}

We define $\dist_\infty(x,\Omega) := \inf_{y \in \Omega} \| x - y \|_{\infty}$ for  $\Omega \subset \R^d$ and $x \in \R^d$.

\begin{lemma}\label{lem:lat}
Let $\mathcal{L} = x_0 +  \Z^d$ be a lattice in $\R^d$, and let $K$ be a symmetric convex body in $\R^d$. For $\eta \geq 1$,
\[
\begin{aligned}
&|\# \{ x \in \cL: \dist_\infty(x,\R^d \setminus K) \geq \eta \} - \mu_d(K) | \leq e(K,\eta)\\
& \# \{ x \in \cL: \dist_\infty(x,\partial K) <  \eta \} \leq 2 e(K,\eta).
\end{aligned}
\]
where, with $C_j(K)$ defined in \eqref{intr_vol_bd:eqn}, we set
\begin{equation}\label{eqn:error_def}
e(K,\eta) := \sum_{j=0}^{d-1} (2 \sqrt{d})^{d-j} \kappa_{d-j} C_j(K) \eta^{d-j}.
\end{equation}

\end{lemma}
\begin{proof}

For any finite set $Y \subset \cL$, let
$D(Y) := \bigcup_{y \in \cL} (y + [-1/2,1/2]^d)$. Note that  $D(X) \subset D(Y)$ if $X \subset Y \subset \cL$, and $\mu_d(D(Y)) = \#Y$.

Consider the disjoint subsets of $\cL$, $X_1 = \{ x \in \cL: \dist_\infty(x,\R^d \setminus K) \geq \eta \}$ and $X_2 = \{ x \in \cL: \dist_\infty(x,\partial K) <  \eta \}$ satisfying
\begin{equation}\label{latlem2:eqn}
X_1 \cup X_2 = \{ x \in \cL:  \dist_\infty(x, K) < \eta \}.
\end{equation}
Our task is to estimate $\# X_1$ and $ \# X_2$.

We establish that, 
\begin{equation}\label{set-inclusions}
K \ominus B (2\sqrt{d}\eta)  \subset D(X_1) \subset D(X_1\cup X_2) \subset K \oplus B(2\sqrt{d}\eta).
\end{equation}

For the proof of the first set inclusion, let $z \in K \ominus B(2\sqrt{d}\eta)$. Then $\dist(z, \R^d \setminus K) \geq 2\sqrt{d}\eta$. Note: $\dist(\cdot,\cdot) \leq \sqrt{d} ~ \dist_\infty(\cdot,\cdot)$. Therefore, $\dist_\infty(z, \R^d \setminus K) \geq 2 \eta$. Any point of $\R^d$ is within $\ell^\infty$ distance of $1/2$ to a point of $\mathcal{L}$. So there exists $p \in  \mathcal{L}$ with $|z-p|_\infty\leq 1/2$. By the triangle inequality, $\dist_\infty(p, \R^d \setminus K) \geq 2\eta - 1/2 > \eta$. Since $p \in \cL$, then $p \in X_1$ by definition of $X_1$. Since $|z - p |_\infty \leq 1/2$, also $z \in p + [-1/2,1/2]^d$, so that $z \in D(X_1)$. This proves the first set inclusion. The second set inclusion is trivial. For the third set inclusion, let $w \in D(X_1 \cup X_2)$ -- then, $w \in q + [-1/2,1/2]^d$ for some $q \in X_1 \cup X_2$. Then $q \in \mathcal{L}$ and $\dist_\infty(q,K) < \eta$ by \eqref{latlem2:eqn}. But $| q - w |_\infty \leq 1/2$. So, by the triangle inequality, $\dist_\infty(w,K) \leq \eta + 1/2 \leq 2 \eta $ -- hence, $\dist(w,K) \leq 2\sqrt{d} \eta$ . Therefore, $w \in K \oplus B(2\sqrt{d}\eta)$, proving the third set inclusion.

From \eqref{set-inclusions}, $\mu_d(D(X)) = \#(X)$, and \eqref{steiner1:eqn}--\eqref{steiner2:eqn}, a volume comparison gives
\[
\begin{aligned}
&\mu_d(K)  - \widetilde{e}(K,\eta)  \leq \mu_d(K \ominus B(2\sqrt{d}\eta))  \leq  \#(X_1) \leq \#(X_1 \cup X_2) \\
& \qquad\qquad\qquad \leq \mu_d(K \oplus B (2\sqrt{d}\eta)) \leq \mu_d(K) + \widetilde{e}(K,\eta), \\
& \mbox{with } \widetilde{e}(K,\eta) :=  \sum_{j=0}^{d-1} (2 \sqrt{d})^{d-j} \kappa_{d-j} V_j(K) \eta^{d-j}.
\end{aligned}
\]
Thus, using that $X_1$ and $X_2$ are disjoint,
\[
| \#(X_1) - \mu_d(K)| \leq \widetilde{e}(K,\eta), \mbox{ and } \#(X_2)  = \#(X_1 \cup X_2) - \#(X_1) \leq 2 \widetilde{e}(K,\eta).
\]
Because $V_j(K) \leq C_j(K)$ (see \eqref{intr_vol_bd:eqn}), we have that $\widetilde{e}(K,\eta) \leq e(K,\eta)$ with $e(K,\eta)$ defined in \eqref{eqn:error_def}. This completes the proof of the result.

\end{proof}

If $K$ is contained in an ellipsoid $ E =  E(r_1,\cdots,r_d)$ then $R_j(K) \leq R_j(E)$ for all $j$, and hence also $C_j(K) \leq C_j(E)$. Given that $R_1(E),\cdots,R_d(E)$ is the nondecreasing rearrangement of $r_1,\cdots, r_d$, from \eqref{intr_vol_bd:eqn}, we obtain:
\[
C_j(K) \leq C_j(E) = 2^j \binom{d}{j} \max_{i_1 < \cdots < i_j} r_{i_1} \cdots  r_{i_j} \leq 2^{j+d} \max_{i_1 < \cdots < i_j} r_{i_1} \cdots  r_{i_j}.
\]
In conjunction with the previous lemma, this estimate gives:
\begin{corollary}\label{cor:lat}
Suppose $K \subset E(r_1,\cdots,r_d)$ for some $r_1,\cdots,r_d > 0$.  For $\eta \geq 1$, 
\[
\begin{aligned}
&|\# \{ x \in \cL: \dist_\infty(x,\R^d \setminus K) \geq \eta \}  - \mu_d(S)| \leq e_0(\eta), \\
&\# \{ x \in \cL: \dist_\infty(x,\partial K) <  \eta \}  \leq 2 e_0(\eta),\\
&\mbox{where  }
e_0(\eta) = 4^d \sum_{j=0}^{d-1} (\sqrt{d})^{d-j}  \kappa_{d-j} \cdot ( \max_{i_1 < \cdots < i_j} r_{i_1} \cdots r_{i_j}) \eta^{d-j}.
\end{aligned}
\]

\end{corollary}

\subsection{Integral estimates on the envelope function}

The   function  $\Psi_a(t) = \exp( - a |t|^{2/3})$, for $a = 1/115$, is used   in \eqref{Fourier_decay_CM:eqn}  to control the Fourier transform of the local sine basis functions. Here we present some integral inequalities for this function.
\begin{lemma} 
There is a constant $C_d > 0$ such that for any $\eta > 1$, $a \in (0,1)$,
\begin{equation}\label{Psi_bd}
\int_{x \in \R^d : |x| > \eta} \Psi_a(|x|) dx \leq C_d a^{-(3/2)d} \Psi_a(\eta / 4).
\end{equation}
\end{lemma}
\begin{proof}
By the Taylor series formula for the exponential function, 
\begin{equation}\label{simple:eqn}
t^K \leq K! e^t \mbox{ for } t > 0, \; K=1,2,3,\cdots.
\end{equation}
Using polar coordinates, if $\sigma_d$ is the surface measure of the $(d-1)$-sphere, then
\[
\int_{x \in \R^d : |x| > \eta} \Psi_a(|x|) dx = \sigma_d \int_{r > \eta} e^{-a r^{2/3}} r^{d-1} dr= \frac{3 \sigma_d}{2}  a^{-\frac{3}{2}d} \int_{t > a \eta^{2/3}} e^{-t} t^{(3d-2)/2} dt,
\]
where the second equality uses the change of variables $t = a r^{2/3}$, $r = a^{-3/2} t^{3/2}$, $dr = (3/2) a^{-3/2} t^{1/2} dt$. Next, applying \eqref{simple:eqn} with $K=3d-2$:
\[
\begin{aligned}
\int_{t > a \eta^{2/3}} e^{-t} t^{(3d-2)/2} dt &\leq \sqrt{(3d-2)!} \int_{t > a \eta^{2/3}} e^{-t/2} dt \\
& = 2 \sqrt{(3d-2)!} e^{-a \eta^{2/3}/2} \leq 2 \sqrt{(3d-2)!} e^{-a (\eta/4)^{2/3}}.
\end{aligned}
\]
Combining the previous lines gives
\[
\int_{x \in \R^d : |x| > \eta} \Psi_a(|x|) dx \leq 3 \sigma_d \sqrt{(3d-2)!} a^{-(3/2)d}  e^{-a (\eta/4)^{2/3}} = C_d a^{-(3/2)d} \Psi_a(\eta/4).
\]
\end{proof}

Next we estimate the sum of $\Psi_a( \pi | x|)$, when $x$ ranges in a subset of a lattice in $\R^d$.

\begin{lemma}\label{lem:lat_sum_Phi}
Let $\mathcal{L}$ be a lattice in $\R^d$ of the form $\mathcal{L} = x_0 + \delta \Z^d = \{ x_0 + \delta \bk : \bk \in \Z^d \}$, with $x_0 \in [0,1)^d$ and $\delta \in (0,1)$. Let $X \subset \mathcal{L}$ be given, and set $D_\delta(X) := \bigcup_{x \in X} ( x + [-\delta/2,\delta/2]^d)$. Then for any $a \in (0,1)$,
\[
\sum_{x \in X} \Psi_a(\pi |x|) \leq C_d \delta^{-d} \int_{D_\delta(X)} \Psi_a( \pi |z|) dz.
\]
\end{lemma}
\begin{proof}
The function $\varphi_a : \R^d \rightarrow \R$, $\varphi_a(z) =  a (\pi | z|)^{2/3}$ is uniformly continuous (in fact, it is H\"{o}lder-$2/3$ continuous), and so $|\varphi_a(z) - \varphi_a(x)| \leq c_d$ whenever $|x-z|_\infty \leq 1$, for some constant $c_d > 0$. Thus,
$\Psi_a(\pi |x|) \leq C_d \Psi_a(\pi |z|)$ for $|x-z|_\infty
\leq 1$, for a constant $C_d > 0$. We integrate this inequality over $z$ in the cube $ x + [-\delta/2,\delta/2]^d$ of Lebesge measure $\delta^d$. We then sum the resulting inequality over $x \in X$, to get
\[
\sum_{x \in X} \Psi_a (\pi |x|) \cdot \delta^d \leq C_d \int_{D_\delta(X)} \Psi_a( \pi |z|) dz.
\]
\end{proof}

\section{Orthonormal wave packet basis for $L^2(Q)$}\label{wave-packet}

Let $\cW$ be the partition of the interval $(0,1)$ defined  in Section \ref{sec:lsb}, and let $Q = [0,1]^d$ be the unit cube in $\R^d$.
We denote by $\cW^{\otimes d}$ the set of all $d$-{\it fold Cartesian products} of intervals in $\cW$, so, $\cW^{\otimes d} = \{ \bL = L_1 \times \cdots \times L_d : L_1,\cdots,L_d \in \cW\}$. We regard the elements $\bL$ of $\cW^{\otimes d}$ as  boxes contained in $Q$.

%We write $\cB(Q)$ to denote the basis on $L^2(Q)$, $Q = [0,1]^d$, defined as follows. 

Define an indexing set $\cI = \cW^{\otimes d} \times \N^d$. Indices in $ \cI$ will be written $\nu = (\bL,\bk)$, with $\bL = L_1 \times \cdots \times L_d \in \cW^{\otimes d}$, $\bk = (k_1,\cdots,k_d) \in \N^d$. 
We define  ``wave packets'' $\psi_{(\bL,\bk)}$ in $L^2(\R^d)$ by
\begin{align}\label{tensor-basis}
\psi_{(\bL,\bk)}(x_1,\cdots,x_d) = \prod_{r=1}^d \phi_{L_r, k_r}(x_r),
\end{align} 
where $\cB(I) = \{\phi_{L, k}\}_{L \in \cW, k \in \N}$ is the local sine orthonormal basis for $L^2(I)$ (see Proposition \ref{basis:prop}).

Set $\cB(Q) = \{ \psi_\nu \}_{\nu \in \cI}$. Then $\cB(Q) = \cB(I)^{\otimes d}$ is the $d$-fold tensor product of the basis $\cB(I)$ for $L^2(I)$.

Identify the $d$-fold tensor power of $L^2(I)$ with $L^2(Q)$. By elementary properties of tensor products of Hilbert spaces, $\cB(Q)$ is an orthonormal basis for $L^2(Q)$. We refer to $\cB(Q) = \{\psi_{(\bL,\bk)}\}_{(\bL,\bk) \in \cI}$ as an ``orthonormal wave packet basis'' for $L^2(Q)$.

\subsection{Fourier decay of wave packets}

Here, we prove a decay bound on the Fourier transform $\widehat{\psi_{(\bL,\bk)}}(\xi)$ 
of a wave packet in $\cB(Q)$. The estimates will follow immediately from the decay estimate \eqref{Fourier_decay_CM:eqn} for the local sine basis functions.

We write $\mathbf{1}$ for the all-ones vector $ (1,\cdots,1)$ in $\R^d$. Given $x  \in \R^d$, $\sigma  \in \{ \pm 1\}^d$, we denote $\sigma \circ x := ( \sigma_1 x_1, \cdots, \sigma_d x_d)$ in $\R^d$. 

\begin{definition}
Given a box $\bL = L_1 \times \cdots \times L_d$ in $\R^d$, define the  map $\tau_{\bL} : \R^d \rightarrow \R^d$ by 
\begin{equation}\label{tauL:def}
\tau_{\bL}(\xi_1, \cdots, \xi_d) := \pi^{-1} ( \delta_{L_1} \xi_1, \cdots,  \delta_{L_d} \xi_d).
\end{equation}
Given $\sigma = (\sigma_1,\cdots,\sigma_d) \in \{\pm 1\}^d$, define the coordinate-wise reflection $\tau_\sigma : \R^d \rightarrow \R^d$ by 
\begin{equation}\label{tausigma:def}
\tau_\sigma(x_1,\cdots,x_d) := \sigma \circ x = (\sigma_1 x_1,\cdots, \sigma_d x_d).
\end{equation}
\end{definition}

Note that $\tau_{\bL}$ is a non-isotropic rescaling map, and $\tau_{\bL}^{-1}$ sends the box $\bL$ to a cube of measure $\pi^d$.

\begin{lemma} \label{lem:wavepacketdecay}

For all  $\nu=(\bL,\bk) \in \cW^{\otimes d} \times \N^d$ and $\xi \in \R^d$, 
\begin{equation} \label{Fourier_decay_CM3:eqn}|\widehat{\psi_{(\bL,\bk)}}(\xi)|^2 \leq 2^d A^{2d} \mu_d(\bL) \sum_{\sigma \in \{\pm 1\}^d} \Psi_{2a}( \pi \cdot |  \tau_{\bL}(\xi) -   \sigma \circ (\bk+ (1/2)\mathbf{1})|).
\end{equation}

\end{lemma}

\begin{proof}
Squaring \eqref{Fourier_decay_CM:eqn}, using  $(s+t)^2 \leq 2 (s^2 + t^2)$, and $\Psi_a(t)^2 = \Psi_{2a}(t)$ (see \eqref{def:Psi}), 
\begin{equation} \label{Fourier_decay_CM2:eqn}|\widehat{\phi_{L,k}}(\xi)|^2 \leq 2A^2 \delta_L \sum_{\sigma \in \{\pm 1\}} \Psi_{2a}( \xi \delta_L - \sigma \pi (k+1/2)).
\end{equation}
Note that
$\widehat{\psi_{(\bL,\bk)}}(\xi) = \prod_{r=1}^d \widehat{\phi_{L_r,k_r}}(\xi_r)$ for $\xi = (\xi_1,\cdots,\xi_d)$,
where $\widehat{\cdot}$ on the LHS is the   Fourier transform on $\R^d$,
and $\widehat{\cdot}$ on the RHS is the   Fourier transform on $\R$. Thus, from \eqref{Fourier_decay_CM2:eqn},
\begin{equation}\label{eqn:tensorbd1}
|\widehat{\psi_{(\bL,\bk)}}(\xi)|^2 \leq 2^d A^{2d} \left( \prod_{j=1}^d \delta_{L_j} \right) \sum_{\sigma = (\sigma_1,\cdots, \sigma_d) \in \{\pm 1\}^d} \prod_{j=1}^d \Psi_{2a}(   \xi_j \delta_{L_j} - \sigma_j \pi ( k_j+1/2)).
\end{equation}
By definition of $\Psi_a$ (see \eqref{def:Psi}),
\begin{equation}\label{help:eqn}
\prod_{j=1}^d \Psi_{2a}(  \omega_j ) =  \exp \left( - 2a ( \sum_{j=1}^d |\omega_j|^{2/3})\right) \leq \exp(- 2 a |\omega|^{2/3}) = \Psi_{2a}(|\omega|)
\end{equation}
for $\omega = (\omega_1,\cdots, \omega_d)$, where we use the inequality $(|t_1|^3 + \cdots + |t_d|^3)^{1/3} \leq |t_1| + \cdots + |t_d|$ (with $t_j = |\omega_j|^{2/3}$).  Note also if $\bL = L_1 \times \cdots \times L_d$ is a box in $\R^d$, then $\prod_j \delta_{L_j} = \mu_d(\bL)$. 

Returning to \eqref{eqn:tensorbd1}, we make the substitution $\mu_d(\bL) = \prod_j \delta_{L_j}$, and then apply the inequality \eqref{help:eqn} for the vector $\omega = \pi \tau_\bL(\xi) - \pi \sigma \circ (\bk + (1/2) \mathbf{1})$, which is given in coordinates by $\omega =  (\omega_1,\cdots,\omega_d)$, $\omega_j = \xi_j\delta_{L_j} - \sigma_j \pi(k_j+1/2)$. This shows
\[
|\widehat{\psi_{(\bL,\bk)}}(\xi)|^2 \leq 2^d A^{2d} \mu_d(\bL) \sum_{\sigma  \in \{\pm 1\}^d}  \Psi_{2a}(  |\pi \tau_\bL(\xi) - \pi \sigma \circ (\bk + (1/2) \mathbf{1})|),
\]
completing the proof of \eqref{Fourier_decay_CM3:eqn}.

\end{proof}

Inequality \eqref{Fourier_decay_CM:eqn} can be interpreted to say that $\widehat{\phi_{L,k}}(\xi)$ is concentrated near two points of $\R$, where $\xi \sim \pm \xi_{L,k}$, with $\xi_{L,k} := \pi (k+1/2)/\delta_L$. By \eqref{tensor-basis}, we have $\widehat{\psi_{(\bL, \bk)}}(\xi) = \prod_{j=1}^d \widehat{\psi}_{L_j,k_j}(\xi_j)$. Therefore, $\widehat{\psi_{(\bL,\bk)}}(\xi)$ is concentrated for $\xi \in \R^d$ close to the $2^d$ many points $(\pm \xi_{L_1,k_1}, \cdots, \pm \xi_{L_d,k_d})$ in $\R^d$.  We now define frequency regions $R_\eta(\bL,\bk)$ (parameterized by $\eta > 1$) where $\widehat{\psi_{(\bL,\bk)}}(\xi)$ is concentrated.

Given a map $\tau: \R^d \rightarrow \R^d$ and a set $S \subset \R^d$, we let $\tau[S] = \{ \tau(x) : x \in S\}$. 

Given $(\bL,\bk) \in \cI = \cW^{\otimes d} \times \N^d$ and $\eta \geq 1$, the set $R_\eta(\bL,\bk) \subset \R^d$ is defined by
\begin{equation}
\label{R_eta_decomp:eqn}
\left\{
\begin{aligned}
&R_\eta(\bL,\bk) = \bigcup_{\sigma \in \{\pm 1\}^d} \tau_\sigma[ R^1_\eta(\bL,\bk)], \mbox{where } \\
&R_\eta^1(\bL,\bk) = \prod_{j=1}^d [\xi_{L_j,k_j} - \eta \delta_{L_j}^{-1}, \xi_{L_j,k_j} + \eta \delta_{L_j}^{-1}], \mbox{ and} \\
&\xi_{L,k} = \pi (k+1/2)/\delta_L.
\end{aligned}
\right.
\end{equation}
We define $\xi_{\bL,\bk} = (\xi_{L_1,k_1},\cdots,\xi_{L_d,k_d})$ in $\R^d$. Observe that $R_\eta(\bL,\bk)$ is the union of $2^d$ many boxes centered at the points $\tau_\sigma(\xi_{\bL,\bk})$ in $\R^d$ ($\sigma \in \{\pm 1\}^d)$, and  $R_\eta(\bL,\bk)$ is coordinate-wise symmetric. The regions are nested with respect to $\eta$, i.e., $R_\eta(\bL,\bk) \subset R_{\eta'}(\bL,\bk)$ for $\eta \leq \eta'$.

\subsection{Main technical result}
\label{mtr:sec}

The following result realizes the assumptions of Lemma \ref{func_anal:lem} for the operator $\cT_r = P_Q B_{S(r)} P_Q$ and the basis $\mathcal{B}(Q) = \{\psi_\nu\}_{\nu \in \cI}$ for $L^2(Q)$.

\begin{proposition}\label{main:prop}
Let $Q = [0,1]^d$.  Let $S \subset \R^d$ be convex and coordinate-wise symmetric (see \eqref{eqn:coord_sym}), with $S \subset B(0,1)$. Let $\delta \in (0,1/2)$ and $r \geq 1$ be given, and define $S(r) := r S = \{rx : x \in S\}$.   Let $\cT_r := \cT_{Q,S(r)} = P_Q B_{S(r)} P_Q$. Let $\{\psi_\nu\}_{\nu\in \cI}$ be the orthonormal wave packet basis for $L^2(Q)$, defined in  \eqref{tensor-basis}.  Then there exists a partition of $\cI = \cI_{low} \cup \cI_{res} \cup \cI_{hi}$, determined by $S, r, \delta$, satisfying   
\[
\sum_{\nu \in \cI_{low}} \| (I-\cT_r) \psi_\nu \|_{2}^2 + \sum_{\nu \in \cI_{hi}} \| \cT_r \psi_\nu \|_{2}^2 
\leq \overline{C}_d \delta r^d
\]
with 
\[
\begin{aligned}
&\# \cI_{res} \leq C_d  \max \{ r^{d-1} \log(r/\delta)^{5/2}, (\log(r/\delta))^{(5/2)d} \}, \mbox{ and} \\
&|\# \cI_{low} - (2 \pi)^{-d} \mu_d(S(r)) | \leq C_d  \max \{ r^{d-1} \log(r/\delta)^{5/2}, (\log(r/\delta))^{(5/2)d} \}. 
\end{aligned}
\]

Here, $C_d$ and $\overline{C}_d \geq 1$ are dimensional constants, determined solely by $d$ (in particular, these constants are independent of the choice of the convex body $S \subset B(0,1)$).
\end{proposition}

The rest of this section is devoted to the proof of the Proposition \ref{main:prop}. 

 \subsubsection{Partitioning the basis}
Here, we start on the proof of Proposition \ref{main:prop}. Given $\delta\in (0,1/2)$, let
\begin{equation}\label{eta_choice:eqn}
\eta := 125 ( a^{-1} \ln(1/\delta))^{3/2},
\end{equation}
with $a=1/115$  as in Proposition \ref{basis:prop}. For this $\eta$,
\begin{equation}\label{eta:eqn}
\Psi_{2a}(\eta/8) \leq \delta.
\end{equation}
Indeed, by \eqref{def:Psi}, $\Psi_{2a}(\eta/8) = \exp( - a \eta^{2/3}/2) = \exp( - 25 \ln(1/\delta)/2) \leq  \delta$.

Associated to the rescaled bandlimiting domain $S(r) = rS$, the value of $\delta$, and $\eta$ (determined by $\delta$), we define a partition $ \{ \cI_{low}, \cI_{res}$, $\cI_{hi} \}$  of $\cI$:
\begin{equation}\label{index_sets:eqn}
\begin{aligned}
&\cI_{low} := \biggl\{ \nu = (\bL,\bk) : \min_{j} \delta_{L_j} > \delta/r, \; R_\eta(\bL,\bk) \subset  S(r) \biggr\}\\
&\cI_{hi} := \biggl\{ \nu = (\bL,\bk) : R_\eta(\bL,\bk) \subset \R^d \setminus S(r) \biggr\} \\
& \qquad\quad \bigcup \biggl\{ \nu = (\bL,\bk) : \min_j \delta_{L_j} \leq \delta/r  \biggr\} \\
&\cI_{res} := \cI \setminus (\cI_{low} \cup \cI_{hi}).
\end{aligned}
\end{equation}

\subsubsection{Energy estimates}

We define the ``energy loss'' for the operator $\cT_r = P_Q B_{S(r)} P_Q$ with respect to the partition \eqref{index_sets:eqn} of the basis $\cB(Q) = \{\psi_\nu\}_{\nu \in \cI}$ by
\[
\mathcal{E} := \sum_{\nu \in \cI_{hi} } \| \cT_r \psi_\nu \|^2 + \sum_{\nu \in \cI_{low}} \| (I-\cT_r) \psi_\nu \|^2.
\]
By Plancherel's theorem, the fact $\psi_\nu$ are supported inside of $Q$, and the fact that $P_Q$ is a contraction on $L^2(\R^d)$, 
\[
\mathcal{E} \leq \sum_{\nu \in \cI_{hi}} \| \widehat{\psi_\nu} \|_{L^2(S(r))}^2 +  \sum_{\nu \in \cI_{low}} \| \widehat{\psi_\nu} \|_{L^2(\R^d \setminus S(r))}^2 =: \cE_{hi} + \cE_{low}.
\]
Here, and in what follows, we write $\| g\|_{L^2(A)} =  (\int_A |g(x)|^2 dx)^{1/2}$ for the $L^2$ norm of $g$ on a measurable set $A \subset \R^d$. 

We now work to bound the two terms on the right-hand side of the above inequality. We write $X \lesssim Y$ to denote  $X \leq C_d Y$ for a dimensional constant $C_d$. 

\begin{lemma}\label{Ehi:lem}
$\cE_{hi} \lesssim \delta r^d$.
\end{lemma}
\begin{proof}
We work to estimate
\[
\cE_{hi} = \sum_{\nu \in \cI_{hi}} \| \widehat{\psi_\nu} \|_{L^2(S(r))}^2.
\]
Recall from \eqref{index_sets:eqn} that $\nu = (\bL,\bk) \in \cI_{hi}$ iff either $\min_j \delta_{L_j} \leq \delta/r$ or $R_\eta(\nu) \subset \R^d \setminus S(r)$. Write
\[
\begin{aligned} 
&\cI_{hi} = \cI_{osc} \cup  \bigcup_{j=1}^d \cI^j_{\partial}, \mbox{ where}\\
&\cI^j_{\partial} := \{ \nu = (\bL, \bk) : \delta_{L_j} \leq \delta/r \}, \quad \mbox{for  } j = 1,\cdots,d;\\
&\cI_{osc} := \{ \nu = (\bL,\bk) : \min_j \delta_{L_j} > \delta/r, \mbox{ and } R_\eta(\nu) \subset \R^d \setminus S(r) \}.
\end{aligned}
\]

Define $Z^j_{\partial}$ ($j \in \{1,\cdots,d\}$) as the sum of  $\| \widehat{\psi_\nu} \|_{L^2(S(r))}^2$ over $\nu \in \cI^j_\partial$, and  $Z_{osc}$ as the sum of  $\| \widehat{\psi_\nu} \|_{L^2(S(r))}^2$ over $\nu \in \cI_{osc}$. Then: 
\[
\cE_{hi} \leq \sum_{j=1}^d Z^j_{\partial} + Z_{osc}.
\]

For $\nu = (\bL,\bk) \in \cI$, we have, by \eqref{tensor-basis},
\[
\widehat{\psi_{(\bL,\bk)}}(\xi_1,\xi_2,\cdots,\xi_d) = \prod_{r=1}^d \widehat{\phi_{L_r,k_r}}(\xi_r).
\]
Since $S(r) = rS$, and $S \subset B(0,1)$, we have $S(r) \subset B(0,r) \subset [-r,r]^d$. Thus,
\begin{equation}\label{Fourier-tensor}
\| \widehat{\psi_\nu} \|_{L^2(S(r)))}^2 \leq \| \widehat{\psi_\nu}\|_{L^2([-r,r]^d)}^2 = \prod_{j=1}^d \| \widehat{\phi_{L_j,k_j}}\|_{L^2([-r,r])}^2.
\end{equation}
By summing \eqref{Fourier-tensor} over all those $\nu = (\bL,\bk) \in \cI_{\partial}^1$, 
for which $\delta_{L_1} \leq \delta/r$, we have:
\begin{equation}\label{Z1_partial}
\begin{aligned}
&Z^1_{\partial}  = \sum_{\nu \in \cI_\partial^1} \| \widehat{\psi_\nu} \|_{L^2(S(r))}^2 \\
&\leq \left(\sum_{(L,k) \in \cW \times \N} \| \widehat{\phi_{L,k}}\|_{L^2([-r,r])}^2 \right)^{d-1} \sum_{(L_1,k_1) \in \cW \times \N : \delta_{L_1} \leq \delta/r} \|\widehat{\phi_{L_1,k_1}}\|_{L^2([-r,r])}^2.
\end{aligned}
\end{equation}

To continue the bound of \eqref{Z1_partial}, we establish:

\begin{claim}[\textbf{Preparatory claim}] 
$\sum_{k \in \N} \| \widehat{\phi_{L,k}} \|_{L^2([-r,r])}^2 \leq \widehat{C} \delta_L r$, for any $L \in \cW$,
for a universal constant $\widehat{C} > 0$.
\end{claim}
\begin{proof}[
\textbf{Proof of preparatory claim}] Recall from \eqref{Fourier_decay_CM2:eqn} that
\[
|\widehat{\phi_{L,k}}(\xi)|^2 \leq 2A^2 \delta_L \sum_{\sigma \in \{\pm 1\}} \Psi_{2a} ( \xi \delta_L - \sigma \pi (k+1/2)).
\]
with $\Psi_a(t) = \exp(- a |t|^{2/3})$. Thus, for fixed $L \in \cW$,
\[
\sum_{k \in \N} \| \widehat{\phi_{L,k}} \|_{L^2([-r,r])}^2 =2 A^2 \delta_L \int_{-r}^r \sum_{\sigma \in \{ \pm 1\}} \sum_{k \in \N} \Psi_{2a}(   \xi \delta_L - \sigma \pi (k+1/2)) d \xi.
\]
For fixed  $\xi \in \R$, we compare a Riemann sum to a Riemann integral, giving:
\[\sum_{\sigma \in \{\pm 1\} } \sum_{k \in \N} \Psi_{2a}(\xi \delta_L - \sigma \pi (k+1/2)) \leq C \int_{t \in \R} \Psi_{2a}(t) dt \leq C_0.\]
Combining this with the previous bound, we establish the claim:
\[
\sum_{k \in \N} \| \widehat{\phi_{L,k}} \|_{L^2([-r,r])}^2 \leq 2 A^2 C_0 \delta_L \int_{-r}^r d \xi = 4 A^2 C_0 \delta_L r.
\]
\end{proof}

Returning to the proof of the lemma, note that $\sum_{L \in \cW} \delta_L = 1$ because $\cW$ is a partition of $(0,1)$. Thus, the preparatory claim implies
\[
\sum_{(L,k) \in \cW \times \N} \| \widehat{\phi_{L,k}} \|_{L^2([-r,r])}^2 \leq \widehat{C} \sum_{L \in \cW} \delta_L r = \widehat{C} r.
\]
Similarly,
\[
\sum_{(L_1,k_1) \in \cW \times \N : \delta_{L_1} \leq \delta/r} \| \widehat{\phi_{L_1,k_1}} \|_{L^2([-r,r])}^2 \leq \widehat{C} \sum_{L_1 \in \cW : \delta_{L_1} \leq \delta/r} \delta_{L_1} r \leq 4 \widehat{C} \delta.
\]
where the latter inequality uses that
$\bigcup_{L\in \cW:  \delta_L\leq \delta/r}L \subset [0,2\delta/r]\cup [1-2\delta/r,1]$ (see \eqref{good_geom1:eqn}). 
So,  $\sum_{L \in \cW : \delta_L \leq \delta/r} \delta_L$ is bounded by the Lebesgue measure of $[0,2\delta/r] \cup [1-2 \delta/r,1]$, which is at most $4 \delta/r$.

Returning to \eqref{Z1_partial}, the above estimates imply
\[
Z^1_\partial \leq (\widehat{C} r)^{d-1} \cdot (4 \widehat{C} \delta) \leq 4 \widehat{C}^d \delta r^{d-1}.
\]
Therefore, by symmetry, 
\begin{equation}\label{Z_partial:eqn}
Z^j_\partial = Z^1_\partial \leq 4 \widehat{C}^d \delta r^{d-1}, \;\; \mbox{for } j=1,\cdots,d.
\end{equation}

It remains to bound $Z_{osc}$. Note that
\begin{equation}\label{Z1_osc}
Z_{osc} = \sum_{\nu \in \cI_{osc}} \| \widehat{\psi_\nu} \|^2_{L^2(S(r))}
\end{equation}
By integrating both sides of \eqref{Fourier_decay_CM3:eqn} over $\xi \in S(r)$, we get:
\[
\| \widehat{\psi_\nu} \|_{L^2(S(r))}^2 \leq 2^d A^{2d} \mu_d(\bL) \sum_{\sigma \in \{\pm 1\}^d} \int_{\xi \in S(r)} \Psi_{2a}( \pi \cdot | \tau_{\bL}(\xi) -    \sigma \circ (\bk+(1/2)\mathbf{1})|) d \xi,
\]
where $\tau_{\bL}$ is the rescaling map on $\R^d$, $\tau_{\bL}(\xi_1, \cdots, \xi_d) = \pi^{-1} ( \delta_{L_1} \xi_1, \cdots,  \delta_{L_d} \xi_d)$, with Jacobian $\det \tau_\bL = \pi^{-d} \mu_d(\bL)$, and $\circ$ is the entrywise product of vectors. By a change of variables,
\begin{equation}\label{psi_nu:eqn}
\| \widehat{\psi_\nu} \|_{L^2(S(r))}^2 \lesssim \int_{\xi \in \tau_{\bL}S(r)} \sum_{\sigma \in \{ \pm 1 \}^d} \Psi_{2a}(\pi \cdot |  \xi -  \sigma \circ (\mathbf{k} + (1/2) \mathbf{1})|) d \xi.
\end{equation}

Recall that $\nu = (\bL,\bk) \in \cI_{osc} \iff \min_j \delta_{L_j} \geq \delta/r$  and $R_\eta(\nu) \subset \R^d \setminus S(r)$. By symmetry of the sets $S(r)$ and $R_\eta(\nu)$ with respect to coordinate reflections $\tau_\sigma$, and by \eqref{R_eta_decomp:eqn}, 
\begin{equation}\label{inclusion_1}
R_\eta(\nu) \subset \R^d \setminus S(r) \iff R^1_\eta(\nu) \subset \R^d \setminus S(r).
\end{equation}
By the form of $R^1_\eta(\nu)$ in \eqref{R_eta_decomp:eqn},
\begin{equation}\label{box_rescale:eqn}
\begin{aligned}
\tau_\bL [R^1_\eta(\nu)] &= \prod_{j=1}^d [(k_j + 1/2) -  \pi^{-1} \eta, (k_j + 1/2) + \pi^{-1} \eta] \\
&= (\bk + (1/2) \mathbf{1}) + [ - \eta/\pi, \eta/\pi]^d.
\end{aligned}
\end{equation}
Also,  $\tau_{\bL} [\R^d \setminus S(r)] = \R^d \setminus \tau_{\bL} [S(r)]$. So, we can apply the rescaling $\tau_{\bL} : \R^d \rightarrow \R^d$ in \eqref{inclusion_1}, and deduce that
\begin{align*}
R_\eta(\nu) \subset \R^d \setminus S(r) &\iff (\bk + (1/2) \mathbf{1}) + [ - \eta/\pi, \eta/\pi]^d \subset \R^d \setminus \tau_{\bL} [S(r)] \\
&\iff \dist_\infty(  \mathbf{k} + (1/2) \mathbf{1}, \tau_{\bL} [S(r)]) \geq  \eta/\pi.
\end{align*}
Returning to the first line of this paragraph, we have proven that
\begin{equation}\label{k_cond_2:eqn}
\nu = (\bL,\bk) \in \cI_{osc} \iff 
\left\{
\begin{aligned}
&\delta_{L_j} \geq \delta/r \mbox{ for all } j\in \{1,\cdots,d\}, \mbox{ and } \\
&\dist_\infty(  \mathbf{k} + (1/2) \mathbf{1}, \tau_{\bL} [S(r)]) \geq  \eta/\pi.
\end{aligned}
\right.
\end{equation}
For $\bL \in \cW^{\otimes d}$, let $\mathcal{K}(\bL) := \{ \bk \in \N^d : \dist_\infty(  \mathbf{k} + (1/2) \mathbf{1}, \tau_{\bL} S(r)) \geq  \eta/\pi \}$. Using \eqref{k_cond_2:eqn} and \eqref{psi_nu:eqn} in \eqref{Z1_osc}, we get
\begin{equation}\label{Z1_osc:eqn}
\begin{aligned}
Z_{osc} \lesssim & \sum_{ \bL \in \cW^{\otimes d}} \sum_{\mathbf{k} \in \mathcal{K}(\bL)} \sum_{\sigma \in \{\pm 1\}^d} \int_{\xi \in \tau_{\bL} S(r)} \Psi_{2a}(|\pi \xi - \pi \sigma \circ (\mathbf{k} + (1/2) \mathbf{1})|) d \xi.
\end{aligned}
\end{equation}
By definition of $\mathcal{K}(\bL)$, for relevant $(\bL, \mathbf{k}, \sigma, \xi)$ as above, $\dist_\infty(\xi, \sigma \circ (\mathbf{k} + (1/2) \mathbf{1})) \geq \eta/\pi$, thus also $| \xi - \sigma \circ (\mathbf{k} + (1/2) \mathbf{1})| \geq \eta/\pi$. Further, the map $(\sigma, \bk) \mapsto \sigma \circ (\mathbf{k} + (1/2) \mathbf{1})$ is injective into the lattice $( \Z + 1/2)^d$. Therefore, by Lemma \ref{lem:lat_sum_Phi}, for any $\xi \in \tau_{\bL} S(r)$:
\[
\begin{aligned}
& \sum_{\mathbf{k} \in \mathcal{K}(\bL)}  \sum_{\sigma \in \{\pm 1\}^d}  \Psi_{2a}(|\pi \xi - \pi \sigma \circ (\mathbf{k} + (1/2) \mathbf{1})|) d \xi \lesssim \int_{\xi' :  | \xi - \xi'| > \eta/\pi} \Psi_{2a}(|\pi \xi - \pi \xi'|) d \xi' \\
& \lesssim \int_{\xi' :  | \xi - \xi'| > \eta} \Psi_{2a}(|\xi - \xi'|) d \xi' = \int_{x : |x| > \eta} \Psi_{2a}( |x|) dx \lesssim \Psi_{2a}(\eta/4),
\end{aligned}
\]
where the last inequality uses \eqref{Psi_bd}. Substituting the above in \eqref{Z1_osc:eqn}, we get:
\[
\begin{aligned}
Z_{osc} \lesssim \Psi_{2a}(\eta/4) \sum_{ \bL \in \cW^{\otimes d}} \mu_d(\tau_{\bL} S(r)) &= \pi^{-d} \Psi_{2a}(\eta/4)  \mu_d(S(r)) \sum_{\bL \in \cW^{\otimes d}}  \mu_d(\bL) \\
& = \pi^{-d} \Psi_{2a}(\eta/4) \mu_d(S(r)),
\end{aligned}
\]
where we used that $\det(\tau_\bL) = \pi^{-d} \mu_d(\bL)$, and that $\cW^{\otimes d}$ is a partition of $ (0,1)^d$, which implies $ \sum_{\bL} \mu_d(\bL) = 1$. From \eqref{eta:eqn}, $\Psi_{2a}(\eta/4) \leq \Psi_{2a}(\eta/8) \leq \delta$. Also, $\mu_d(S(r)| = \mu_d(S) r^d \leq \mu_d(B(0,1)) r^d$, so
\[
Z_{osc} \lesssim \delta r^d.
\]
Combining this with our earlier bound \eqref{Z_partial:eqn} on $Z_\partial^j$, we obtain
\[
\cE_{hi} \leq \sum_{j=1}^d Z_{\partial}^j + Z_{osc} \lesssim \delta r^{d-1} +  \delta r^d \lesssim \delta r^d.
\]
This completes the proof of the lemma.
\end{proof}

\begin{lemma}\label{Elow:lem}
$\cE_{low} \lesssim \delta r^d$.
\end{lemma}
\begin{proof}

We work to estimate
\[
\cE_{low} = \sum_{\nu \in \cI_{low}} \| \widehat{\psi_\nu} \|_{L^2(\R^d \setminus S(r))}^2.
\]
As in the proof of  \eqref{psi_nu:eqn}, we integrate both sides of \eqref{Fourier_decay_CM3:eqn} over $\xi \in \R^d \setminus S(r)$, and make a change of variables, to get:
\begin{equation}\label{psi_nu_comp:eqn}
\| \widehat{\psi_\nu} \|_{L^2(\R^d \setminus S(r))}^2 \lesssim \int_{\xi \in \R^d \setminus \tau_{\bL} S(r)} \sum_{\sigma \in \{ \pm 1 \}^d} \Psi_{2a}(|\pi \xi -  \pi \sigma \circ (\mathbf{k} + (1/2) \mathbf{1})|) d \xi,
\end{equation}
where $\tau_{\bL}$ is the rescaling map on $\R^d$ defined by $\tau_\bL(\xi) = (\pi^{-1} \delta_{L_1} \xi_1, \cdots, \pi^{-1} \delta_{L_d} \xi_d)$.

Recall from \eqref{index_sets:eqn} that $\nu = (\bL,\bk) \in \cI_{low} \iff \min_j \delta_{L_j} > \delta/r$, and $ R_\eta(\bL,\bk) \subset S(r)$. Now, replicating the proof of \eqref{k_cond_2:eqn},
\begin{equation}\label{Ilow_cond:eqn}
\nu = (\bL,\bk) \in \cI_{low} \iff 
\left\{
\begin{aligned}
&\delta_{L_j} > \delta/r \mbox{ for all } j\in \{1,\cdots,d\}, \mbox{ and } \\
&\dist_\infty(  \mathbf{k} + (1/2) \mathbf{1}, \R^d \setminus \tau_{\bL} [S(r)]) \geq  \eta/\pi.
\end{aligned}
\right.
\end{equation}
For $\bL \in \cW^{\otimes d}$, let $\overline{\cK}(\bL) := \{ \bk \in \N^d : \dist_\infty(\bk + (1/2) \mathbf{1}, \R^d \setminus  \tau_{\bL} S(r)) \geq \eta/\pi \}$. Thus,
\begin{equation}\label{E_low1:eqn}
\begin{aligned}
\cE_{low} \lesssim \sum_{\bL \in \cW^{\otimes d}} \sum_{\mathbf{k} \in \overline{\cK}(\bL)}  \sum_{\sigma \in \{\pm 1\}^d} \int_{\xi \in \R^d \setminus \tau_{\bL} S(r)} \Psi_{2a}(|\pi \xi - \pi \sigma \circ ( \mathbf{k} + (1/2) \mathbf{1})|) d \xi.
\end{aligned}
\end{equation}
For $(\bk, \sigma)$ as above, the points $\xi_{\sigma,\bk} = \sigma \circ ( \mathbf{k} + (1/2) \mathbf{1})$ satisfy $\xi_{\sigma,\bk} \in (\Z + 1/2)^d$, while $\dist_\infty(\xi_{\sigma,\bk},  \R^d \setminus \tau_{\bL} S(r)) \geq \eta/\pi$, with the last inequality using that $\bk \in \overline{\cK}(\bL)$ and the coordinate-symmetry of $\R^d \setminus \tau_{\bL} S(r)$. In particular, $\dist_\infty(\xi , \xi_{\sigma,\bk}) \geq \eta/\pi$ for $\xi \in \R^d \setminus \tau_{\bL} S(r)$. The triangle inequality implies $\xi_{\sigma, \bk} + [-1/2,1/2]^d \subset \{ \xi' \in \tau_{\bL} S(r): \dist_\infty(\xi,\xi') \geq \eta/\pi - 1\} \subset \{ \xi'  \in \tau_{\bL} S(r): |\xi - \xi'| \geq \eta/\pi - 1\}$. Note $\eta \geq 2\pi$, by definition, so $\eta/\pi - 1 \geq \eta/(2 \pi)$.
Consequently, by Lemma \ref{lem:lat_sum_Phi},  for any $\xi \in \R^d \setminus  \tau_{\bL} S(r)$,
\[
\begin{aligned}
\sum_{\mathbf{k} \in \overline{\mathcal{K}}(\bL)}  \sum_{\sigma \in \{\pm 1\}^d}  \Psi_{2a}(|\pi \xi - \pi \xi_{\sigma,\bk}|)  &\lesssim \int_{\xi' \in \tau_{\bL} S(r)  : |\xi - \xi'| \geq \eta/(2\pi)} \Psi_{2a}(|\pi \xi - \pi \xi'|) d \xi'.
\end{aligned}
\]
Substituting this inequality in \eqref{E_low1:eqn},
\[
\cE_{low} \lesssim \sum_{\bL \in \cW^{\otimes d}} \int_{(\xi, \xi') \in (\R^d \setminus \tau_{\bL} S(r)) \times \tau_{\bL} S(r) : |\xi - \xi' | \geq \eta/(2 \pi)} \Psi_{2a}(|\pi \xi- \pi \xi'|) d \xi d  \xi'
\]
By the change of variables $\xi'' =  \pi( \xi- \xi')$, we change the integral into a $d \xi' d \xi''$ integral integrated over the larger set of all $(\xi',\xi'') \in \tau_{\bL} S(r) \times \R^d$ with $|\xi''| \geq \eta/2$. Thus,
\[
\begin{aligned}
\cE_{low} &\lesssim \sum_{\bL \in \cW^{\otimes d}} \int_{\xi' \in \tau_{\bL} S(r)} \int_{ |\xi''| \geq \eta/2} \Psi_{2a}( |\xi''|) d \xi'' d \xi' \\
&= \sum_{\bL \in \cW^{\otimes d}} \mu_d( \tau_\bL S(r)) \int_{|\xi''| \geq \eta/2} \Psi_{2a}(|\xi''|).
\end{aligned}
\]
First using \eqref{Psi_bd}, $\mu_d(S(r)) \leq \mu_d(B(r)) \lesssim r^d$, and that the Jacobian of the transformation $\tau_\bL$ is $\det \tau_\bL = \pi^{-d} \mu_d(\bL)$, and next using \eqref{eta:eqn} and $\sum_{\bL \in \cW^{\otimes d}} \mu_d(\bL) = 1$,
\[
\cE_{low} \lesssim \Psi_{2a}(\eta/8) r^d \sum_{\bL \in \cW^{\otimes d}} \mu_d( \bL) \lesssim \delta r^d.
\]
\end{proof}

\begin{lemma}\label{count1:lem}
There exists a dimensional constant $C_d$ such that 
\[
\# ( \cI_{res}) \leq C_d \max \left\{ r^d \log(r/\delta)^{5/2}, \log(r/\delta)^{(5/2)d}\right\}.
\]
\end{lemma}
\begin{proof}
If  $\nu = (\bL,\bk) \in \cI_{res}$ then $\nu \notin \cI_{hi}$. By definition of $\cI_{hi}$ (see \eqref{index_sets:eqn}), 
\begin{equation}\label{Lcond:eqn}
\nu \in \cI_{res} \implies \delta_{L_j} > \delta/r  \; \forall j.
\end{equation}
Because $\cW$ is a family of subintervals of $[0,1]$, we have $\delta_L \leq 1$ for all $L \in \cW$. By definition of $\cW$, there are exactly $2$ intervals $L \in \cW$ satisfying  $\delta_L = 2^{- k}$ for each $k \geq 2$. Therefore, 
\begin{equation} \label{Wh_count:eqn}
\# \{ L \in \cW : \delta_L \geq \delta/r \} \leq 2 \log(r/\delta).
\end{equation}

If $\nu \in \cI_{res}$ then $\nu \notin \cI_{low}$ and $\nu \notin \cI_{hi}$, and thus   $R_\eta(\nu) \not\subset S(r)$ and $R_\eta(\nu) \not\subset \R^d \setminus S(r)$ (see \eqref{index_sets:eqn}); thus, $\nu \in \cI_{res} \implies R_\eta(\nu) \cap \partial S(r) \neq \emptyset$. Note, $R_\eta(\nu)$ and $S(r)$ are coordinate-wise symmetric, and further, in \eqref{R_eta_decomp:eqn}, $R_\eta(\nu) = \bigcup_{j=1}^{2^d} R_\eta^j(\nu)$, where $R_\eta^2(\nu),\cdots,R_\eta^{2^d}(\nu)$ are given by applying coordinate-wise reflections to $R_\eta^1(\nu)$. 
\[
\nu \in \cI_{res} \implies R_\eta^1(\nu) \cap \partial S(r) \neq \emptyset.
\]
We  apply the rescaling map $\tau_\bL$ on $\R^d$, $\tau_\bL(\xi_1,\cdots,\xi_d) = \pi^{-1} (\delta_{L_1} \xi_1, \cdots, \delta_{L_d} \xi_d)$, in the condition above. Using \eqref{box_rescale:eqn},
\begin{equation}\label{star2}
\begin{aligned}
\nu \in \cI_{res} &\implies ( ( \bk + (1/2) \mathbf{1}) + [-\eta/\pi, \eta/\pi]^d) \cap  \partial ( \tau_\bL S(r)) \neq \emptyset \\
&\iff \dist_\infty( \mathbf{k} + (1/2) \mathbf{1},  \partial (\tau_{\bL} S(r)) ) \leq \eta/\pi.
\end{aligned}
\end{equation}

Let $N(\bL) := \# \{ \bk \in \N^d : \dist_\infty( \mathbf{k} + (1/2) \mathbf{1},  \partial (\tau_{\bL} S(r)) ) \leq \eta/\pi \}$ for fixed $\bL \in \cW^{\otimes d}$. Then, by \eqref{Lcond:eqn} and \eqref{star2},
\begin{equation}
\label{Ires_bd}
\#(\cI_{res}) \leq \sum_{\bL \in \cW^{\otimes d} : \delta_{L_j} > \frac{\delta}{r} \forall j } N (\bL).
\end{equation}

By assumption, $S \subset B(0,1)$, so $S(r) \subset B(r)$, so
\begin{equation}\label{rescaled_set_cont:eqn}
\tau_{\bL} S(r) \subset \tau_{\bL}B(r) = E( \pi^{-1} r \delta_{L_1}, \cdots,  \pi^{-1} r  \delta_{L_d} ), \end{equation}
where $E(r_1, \cdots, r_d)$ is the ellipsoidal set consisting of all $(\xi_1, \cdots, \xi_d) \in \R^d$ satisfying $(\xi_1/r_1)^2 + \cdots + (\xi_d/r_d)^2 \leq 1$.  Thus, by Corollary \ref{cor:lat}, 
\[
\begin{aligned}
N(\bL) & \leq  2 \cdot 4^d \sum_{j=0}^{d-1} (\sqrt{d})^{d-j}  \kappa_{d-j} \max_{i_1 < \cdots < i_j} (\pi^{-1} r \delta_{L_{i_1}}) \cdots (\pi^{-1} r \delta_{L_{i_j}})  (\eta/\pi)^{d-j} \\
&\lesssim \sum_{j=0}^{d-1} r^j \eta^{d-j} \max_{i_1 < \cdots < i_j} \delta_{L_{i_1}} \cdots \delta_{L_{i_j}},
\end{aligned}
\]
with $\kappa_\ell$ the volume of the unit ball in $\R^\ell$. Returning to \eqref{Ires_bd},
\begin{equation}\label{Ires_bd:eqn}
\#(\cI_{res}) \lesssim \sum_{\bL \in \cW^{\otimes d} : \delta_{L_j} > \frac{\delta}{r} \forall j } \sum_{j=0}^{d-1} r^j \eta^{d-j} \max_{i_1 < \cdots < i_j} \delta_{L_{i_1}} \cdots \delta_{L_{i_j}} =: T_0.
\end{equation}
We estimate $T_0$ as follows. By symmetry, 
\[
T_0 \leq d! \sum_{\bL : \delta_{L_1} \geq \cdots \geq \delta_{L_d} \geq \delta/r} \sum_{j=0}^{d-1} r^j \eta^{d-j} \max_{i_1 < \cdots < i_j}  \delta_{L_{i_1}} \cdots  \delta_{L_{i_j}}.
\]
Here, the factor of $d!$ accounts for all possible permutations of the ordering of the $d$ interval lengths $\delta_{L_1},\cdots,\delta_{L_d}$. The maximum in the above sum is equal to $\delta_{L_1} \cdots \delta_{L_j}$. We sum over a larger collection of $\bL$, to give
\[
T_0 \lesssim \sum_{\bL: \delta_{L_{j+1}}, \cdots, \delta_{L_d} \geq \delta/r} \sum_{j=0}^{d-1} r^j \eta^{d-j}   \delta_{L_1} \cdots  \delta_{L_j}.
\]
The summands on the RHS above do not depend on $L_{j+1},\cdots,L_d$. The number of $(L_{j+1},\cdots,L_d) \in \cW^{d-j}$ satisfying $\delta_{L_{j+1}}, \cdots , \delta_{L_d} \geq \delta /r$ is at most $(2 \log(r /\delta))^{d-j}$, thanks to \eqref{Wh_count:eqn}. Thus,
\[
T_0 \lesssim  \sum_{j=0}^{d-1} r^j \eta^{d-j} (\log(r/\delta))^{d-j} \sum_{L_1,\cdots,L_j \in \cW} \delta_{L_1} \cdots \delta_{L_j}.
\]
Note  $\sum_{L_1,\cdots, L_j \in \cW} \delta_{L_1} \cdots \delta_{L_j} = \left( \sum_{L \in \cW} \delta_L \right)^j$. Also, $\sum_{L \in \cW} \delta_L = 1$, since $\cW$ is a partition of $(0,1)$. By definition, $\eta = 125 (a^{-1} \ln(1/\delta))^{3/2} \leq C \log(1/\delta)^{3/2}$, so
\[
T_0 \lesssim  \sum_{j=0}^{d-1} r^j (\log(r/\delta))^{(5/2)(d-j)}.
\]
Note that $\sum_{j=0}^{d-1} r^{j} (\log(r/\delta))^{(5/2)(d-j)}$ is a geometric series. Its terms are monotone. Thus, the sum of the series is at most $d$ times the maximum of the first term ($j=0$) and last term ($j=d-1$). We deduce that
\begin{equation}\label{T0_bd:eqn}
T_0 \lesssim \max \{ r^{d-1} \log(r/\delta)^{5/2}, \log(r/\delta)^{(5/2)d}\}.
\end{equation}
From \eqref{Ires_bd:eqn}, $\#(\cI_{res}) \lesssim T_0$. This concludes the proof of the lemma.
\end{proof}

\begin{lemma}\label{count2:lem}
There exists a dimensional constant $C_d$ such that 
\[
| \# \cI_{low} - (2 \pi)^{-d} \mu_d(S(r))| \leq C_d \max \left\{ r^d \log(r/\delta)^{5/2}, \log(r/\delta)^{(5/2)d} \right\}.
\]
\end{lemma}

\begin{proof}
We apply \eqref{Ilow_cond:eqn} to bound  $\#\cI_{low}$. 

For  $\bL \in \cW^{\otimes d}$, let $\cM(\bL) := \# \{ \bk \in \N^d : \dist_\infty(\bk + (1/2) \mathbf{1}, \R^d \setminus \tau_{\bL} S(r) ) \geq \eta/\pi \}$. By symmetry,  $\cM(\bL) = 2^{-d} \cM_0(\bL)$, where $\cM_0(\bL) = \# \{ x \in  (\Z + 1/2)^d : \dist_\infty(x, \R^d \setminus \tau_{\bL} S(r) ) \geq \eta/\pi\}$. Due to \eqref{rescaled_set_cont:eqn}, we can apply Corollary \ref{cor:lat}, giving
\[
|\cM_0(\bL) -   \mu_d(\tau_{\bL} S(r)) | \leq C_d \sum_{j=0}^{d-1} \eta^{d-j} \max_{i_1 < \cdots < i_j} (r \delta_{L_{i_1}}) \cdots (r \delta_{L_{i_j}}).
\]
The Jacobian of the map $\tau_{\bL}$ is $\det \tau_{\bL} = \pi^{-d} \delta_{L_1} \cdots \delta_{L_d} = \pi^{-d} \mu_d(\bL)$. Thus, by multiplying the previous inequality by $2^{-d}$, and using the Jacobian change of volume formula, 
\begin{equation}\label{M_bd:eqn}
|\cM(\bL) -  (2 \pi)^{-d} \mu_d(S(r)) \mu_d(\bL)  | \lesssim \sum_{j=0}^{d-1} r^j \eta^{d-j} \max_{i_1 < \cdots < i_j}  \delta_{L_{i_1}} \cdots  \delta_{L_{i_j}}.
\end{equation}

By  \eqref{Ilow_cond:eqn} and the definition of $\cM(\bL)$, 
\begin{equation} \label{I_low_decomp:eqn}
\# \cI_{low} = \sum_{\bL \in \cW^{\otimes d} : \delta_{L_j} > \frac{\delta}{r} \forall j} \cM(\bL).
\end{equation}

Observe that
\[
\left|\sum_{\bL : \delta_{L_j} > \frac{\delta}{r} \forall j} \mu_d(\bL) - 1 \right| = \left|\sum_{\bL : \delta_{L_j} > \frac{\delta}{r} \forall j} \mu_d(\bL) - \mu_d((0,1)^d) \right|  = \sum_{\bL : \min_j \delta_{L_j} \leq \frac{\delta}{r} } \mu_d(\bL),
\]
where the last equality uses that $\cW^{\otimes d} = \{\bL \}$ is a partition of $ (0,1)^d$. Now, if $\min_j \delta_{L_j} \leq \delta/r$, then $\bL = L_1 \times \cdots \times L_d$ is contained in the region $[0,1]^d \setminus [2 \delta/r, 1-2 \delta/r]^d$; recall, by \eqref{good_geom1:eqn}, $\delta \in \cW, \delta_L \leq \delta_0 \implies L \subset [0,2 \delta_0] \cup [1-2 \delta_0,1]$. Hence,
\[
\sum_{\bL : \min_j \delta_{L_j} \leq \frac{\delta}{r}} \mu_d(\bL) \leq \mu_d([0,1]^d \setminus [2 \delta/r, 1-2 \delta/r]^d) \leq 2d (2 \delta/r) = 4d \delta/r.
\]
Thus,
\[
\left|\sum_{\bL : \delta_{L_j} > \frac{\delta}{r} \forall j} \mu_d(\bL) - 1 \right| \leq 4 d \delta/r.
\]

Now, by summing \eqref{M_bd:eqn} and using \eqref{I_low_decomp:eqn}, as well as the above line, we get
\begin{equation}\label{I_low_bd:eqn}
\begin{aligned}
| \# \cI_{low} - & (2\pi)^{-d} \mu_d(S(r))| \leq  \underbrace{(4d \delta/r) (2\pi)^{-d} \mu_d(S(r))}_{T_1}  \\
&+ C_d \underbrace{ \sum_{\bL : \delta_{L_j} > \frac{\delta}{r} \forall j} \sum_{j=0}^{d-1} r^j \eta^{d-j} \max_{i_1 < \cdots < i_j}  \delta_{L_{i_1}} \cdots  \delta_{L_{i_j}}}_{T_0}
\end{aligned}
\end{equation}
The term $T_1$ is controlled as follows, using that $S(r) \subset B(r)$,
\[
T_1 = (4d \delta/r) (2\pi)^{-d} \mu_d(S(r)) \lesssim (\delta/r) r^d \lesssim r^{d-1}.
\]
The term $T_0$ is the same term defined in \eqref{Ires_bd:eqn}, and bounded in \eqref{T0_bd:eqn}.
Using these bounds in \eqref{I_low_bd:eqn}, 
\[
| \# \cI_{low} - (2 \pi)^{-d} \mu_d(S(r))| \lesssim \max \{ r^{d-1} \log(r/\delta)^{5/2}, \log(r/\delta)^{(5/2)d} \}.
\]
This completes the proof of the lemma.
\end{proof}

From Lemmas \ref{Ehi:lem} and \ref{Elow:lem}, $\cE \leq \cE_{hi} + \cE_{low} \leq C_d \delta r^d$. From Lemmas \ref{count1:lem} and \ref{count2:lem},
\[
\begin{aligned}
&\# \cI_{res} \leq C_d \max \{ r^{d-1} \log(r/\delta)^{5/2}, \log(r/\delta)^{(5/2)d} \} \\
& | \# \cI_{low} - (2 \pi)^{-d} \mu_d(S(r))| \leq C_d  \max \{ r^{d-1} \log(r/\delta)^{5/2}, \log(r/\delta)^{(5/2)d} \}.
\end{aligned}
\]
This completes the proof of Proposition \ref{main:prop}

\section{Proofs of the Main Results}\label{sec:proof-of-main-thm-1}

\subsection{Proof of Theorem \ref{mainthm:cube_convex}}
\label{proof-of-mainthm}

Fix a convex set $S \subset \R^d$ that is coordinate-wise symmetric. Let $Q = [ 0,1]^d$. Let $\epsilon \in (0,1/2)$ and $r \geq 1$. Set $S(r) = \{ r x : x \in S \}$. Consider the operator $\cT_r = P_Q B_{S(r)} P_Q$ on $L^2(\R^d)$. Let $\lambda_k(\cT_r)$ denote the nonzero eigenvalues of $\cT_r$ ($k \geq 0$). These are identical to the eigenvalues  $\lambda_k(Q,S(r))$ for the operator $\widetilde{T}_{Q,S(r)} = B_{S(r)} P_Q B_{S(r)}$; see Remark \ref{remk:SSLO_alt}.

Let $C_d$, $\overline{C}_d \geq 1$ be as in Proposition \ref{main:prop}. Define $\delta = \frac{\epsilon^2}{\overline{C}_d r^d} \in (0,1/4]$. We apply Proposition \ref{main:prop} to produce a partition $\cI = \cI_{low} \cup \cI_{res} \cup \cI_{hi}$ of the index set for the basis $\cB(Q) = \{ \psi_\nu\}_{\nu \in \cI}$, such that
\begin{equation}\label{E_loss:eqn}
\sum_{\nu \in \cI_{low}} \| (I-\cT_r) \psi_\nu \|_{2}^2 + \sum_{\nu \in \cI_{hi}} \| \cT_r \psi_\nu \|_{2}^2 \leq \overline{C}_d \delta r^d  = \epsilon^2
\end{equation}
and 
\[
\begin{aligned}
\# (\cI_{res}) &\leq C_d \max \{ r^{d-1} \log(r/\delta)^{5/2}, (\log(r/\delta))^{(5/2)d} \} \\
| \# (\cI_{low}) - (2 \pi)^{-d} \mu_d(S(r))| &\leq C_d  \max \{ r^{d-1} \log(r/\delta)^{5/2}, (\log(r/\delta))^{(5/2)d} \}.
\end{aligned}
\]
Note that 
\[
\log(r/\delta) = \log( \overline{C}_d r^{d+1}/\epsilon^2)) \leq \log(\overline{C}_d) + (d+1) \log(r/\epsilon) \leq C_d \log(r/\epsilon).
\]
Therefore, 
\[
\begin{aligned}
\# (\cI_{res}) &\leq C_d \max \{ r^{d-1} \log(r/\epsilon)^{5/2}, (\log(r/\epsilon))^{(5/2)d} \} \\
| \# (\cI_{low)} - (2 \pi)^{-d} \mu_d(S(r))| &\leq C_d  \max \{ r^{d-1} \log(r/\epsilon)^{5/2}, (\log(r/\epsilon))^{(5/2)d} \}.
\end{aligned}
\]
Owing to \eqref{E_loss:eqn}, we can apply Lemma \ref{func_anal:lem}. We learn that, for any $\epsilon \in (0,1/2)$,
\[
\begin{aligned}
|\# \{ k : \lambda_k(\mathcal{T}_r) > \epsilon \} - \#(\cI_{low})| &\leq \#(\cI_{res}) \\
\# \{ k : \lambda_k(\mathcal{T}_r) \in( \epsilon,1-\epsilon) \} &\leq \#(\cI_{res}).
\end{aligned}
\]
Applying the preceding bounds on $\#(\cI_{low})$ and $\#(\cI_{res})$, we obtain the bounds \eqref{eig_clust:eqn1} and \eqref{eig_clust:eqn2}  in Theorem \ref{mainthm:cube_convex}.

  \begin{remark}\label{imp:detailed rem} 
  Following our discussion in Remark \ref{imp:rem}, 
  to improve the bounds \eqref{eig_clust:eqn1}--\eqref{eig_clust:eqn3} in Theorem \ref{mainthm:cube_convex},  we replace the cutoff function $v(x)$ in \eqref{c0} by  $v_m(x) := e^{-(1-x^2)^{-m}}$ for $x\in (-1,1)$ and zero elsewhere, which belongs to the Gevrey class $G^{1+\frac{1}{m}}$ introduced in Section \ref{sec:Gevrey}. One can check that the functions $\psi$ and $s$ constructed in \eqref{cutoff-function}, \eqref{s_theta_defn} are then in $G^{1+\frac{1}{m}}$. For such $s$, the Fourier decay condition \eqref{uniformbd1_a} for the local sine basis  is satisfied with exponent $\frac{2}{3}$ replaced by $\frac{m}{m+1}$. The proof of this improved Fourier decay condition requires a variant of Proposition \ref{prop:Gevrey_PW} for the Gevrey class $G^\delta$ with $\delta = 1+\frac{1}{m} >1$; for instance, one can apply Proposition 2.8 of \cite{DH} to obtain this bound. By instead fixing the parameter $\eta \sim \ln(1/\delta)^{\frac{m+1}{m}}$ in \eqref{eta_choice:eqn}, one can adapt the arguments of Section \ref{mtr:sec} and derive the improved bounds in \eqref{eig_clust:eqn1}--\eqref{eig_clust:eqn3}.
  \end{remark}

\subsection{Proof of Theorem \ref{tensor:propA}}
\label{proof-of-propA}

Recall, given domains $Q,S \subset \R^d$, the operator $\mathcal{T}_{Q,S} : L^2(\R^d) \rightarrow L^2(\R^d)$ is defined by $\mathcal{T}_{Q,S} = P_Q B_S P_Q$. In particular, if $I, J \subset \R$ are intervals then $\cT_{I,J} : L^2(\R) \rightarrow L^2(\R)$. We establish the following
\begin{proposition}\label{tensor:prop1}
Let $Q = I^d$ and $S = J^d$ be cubes in $\R^d$. Let $ \{ \psi_l\}_{l \geq 0} $ be an orthonormal basis for $L^2(I)$ consisting of the eigenfunctions of the operator $\cT_{I,J}$, and let $\{\lambda_l \}_{l \geq 0}$ be the sequence of associated positive eigenvalues, written in descending order.
Then the set
$$\left\{\Psi_{\vec{\ell}}(x) =  \prod_{k=1}^d \psi_{\ell_k}(x_k) : \vec{\ell} = (\ell_1,\cdots,\ell_d) \in \N^d\right\}
$$
is an orthonormal basis for $L^2(Q)$,  consisting of the eigenfunctions of $\cT_{Q,S}$, with associated eigenvalues
$$
\left\{ \lambda_{\vec{\ell}} =  \prod_{k=1}^d \lambda_{\ell_k} : \vec{\ell} = (\ell_1,\cdots,\ell_d) \in \N^d\right\}.
$$
\end{proposition}
\begin{proof}
If $\cB$ is an orthonormal basis for a Hilbert space $\mathcal{H}$, then $\cB^{\otimes d} := \{ \otimes_{j=1}^d \psi_j : \psi_j \in \cH\}$ is an orthonormal basis for $\overline{\bigotimes_{k=1}^d \mathcal{H}}$. Note that $L^2(Q) =\overline{ \bigotimes_{k=1}^d L^2(I)}$. Therefore,  $\{ \Psi_{\vec{\ell}} \}$ is an orthonormal basis for $L^2(Q)$.

It remains to verify that $\cT_{Q,S} \Psi_{\vec{\ell}} = \lambda_{\vec{\ell}} \Psi_{\vec{\ell}} $.  
Recall that $\cT_{Q,S} = P_QB_SP_Q$ as introduced in Section \ref{sec:prelim}.  
We have the kernel representation for $\cT_{Q,S}$:
\[
\cT_{Q,S} \Psi(x) = \int_{\R^d} K_{Q,S}(x,y) \Psi(y) dy \qquad ( \Psi \in L^2(\R^d)).
\] 
with 
\[
K_{Q,S}(x,y) = \chi_Q(x) \chi_S(y) \pi^{-d} \prod_{k=1}^d \frac{\sin (W(x_k-y_k))}{x_k - y_k},
\]
where the term $\pi^{-d}$ times the product of sinc functions is just the inverse fourier transform of $\chi_{[-W,W]^d}(\xi)$. Given that $Q = I^d$ we have $\chi_Q(x) = \prod_k \chi_I(x_k)$, and similarly for $\chi_S(y) = \prod_k \chi_J(y_k)$ so
\[
K_{Q,S}(x,y) = \prod_{k=1}^d K_{I,J}(x_k,y_k),
\]
with $K_{I,J}$ defined as before. Hence, using the definition of $\Psi_{\vec{\ell}}$, and the eigenproperty of $\psi_{\ell_k}$,
\[
\cT_{Q,S} \Psi_{\vec{\ell}}(x) = \prod_{k=1}^d \int_{\R} K_{I,J}(x_k,y_k) \psi_{\ell_k}(y_k) dy = \prod_{k=1}^d \lambda_{\ell_k} \psi_{\ell_k}(x_k)
\]
So, $\cT_{Q,S} \Psi_{\vec{\ell}} = \lambda_{\vec{\ell}} \Psi_{\vec{\ell}}$, as claimed.

\end{proof}

Combining Theorem \ref{Karnik}   with Landau's result \eqref{landau-half}, we obtain the following. 

\begin{proposition}\label{1d:prop}
Let $W \geq 2 \pi$. Let $\lambda_k = \lambda_k(W)$, $k \geq 0$, be the positive eigenvalues of the operator $\cT_{I,J} : L^2(\R) \rightarrow L^2(\R)$, where $I=[0,1]$ and $J=[-W,W]$. Then, for any $\gamma \in (0,1)$,
\[
|\# \{ k : \lambda_k > \gamma \} - W/\pi| \leq (2 / \pi^2) \log( 50W/\pi + 25) \log(5/(\gamma(1-\gamma))) + 9.
\]
\end{proposition}
\begin{proof}
We split the proof into cases depending on whether $\gamma \leq 1/2$ or $\gamma  >  1/2$.

{\bf Case 1:} 
Let $\gamma \leq  1/2$. Then write
\[
\# \{ k : \lambda_k > \gamma \} = \# \{ k : \lambda_k \in [1/2,1) \} +  \# \{ k : \lambda_k \in (\gamma,1/2)\}.
\]
The first term is between $N$ and $N+2$, with $N = \lfloor W/\pi \rfloor$, by Landau's result  \eqref{landau-half}, and the fact that the eigenvalues $\lambda_k = \lambda_k(W)$ are strictly decreasing (see \cite{Karnik21}). Therefore, the first term is between $W/\pi -1$ and $W/\pi + 2$. The second term is at most $\# \{ k : \lambda_k \in (\gamma, 1 - \gamma) \} \leq (2 / \pi^2) \log( 50W/\pi + 25) \log(5/(\gamma(1-\gamma))) + 7$,  due to Theorem \ref{Karnik}. Therefore,
\[
|\# \{ k : \lambda_k > \gamma \} - W/\pi| \leq  (2 / \pi^2) \log( 50W/\pi + 25) \log(5/(\gamma(1-\gamma))) + 9.
\]
 
 {\bf Case 2:} 
Let $\gamma  >  1/2$. Then the analysis is similar to Case 1, except one uses
\[
\# \{ k : \lambda_k > \gamma \} = \# \{ k : \lambda_k \in [1/2,1) \} -  \# \{ k : \lambda_k \in [1/2, \gamma] \},
\]
and then applies \eqref{landau-half} and Theorem  \ref{Karnik} with $\epsilon = 1-\gamma \in (0,1/2)$. We spare the details. This completes the proof of Proposition \ref{1d:prop}.
\end{proof}

Next we give the proof of Theorem \ref{tensor:propA}.

Let $Q = [0,1]^d$ and $S = [-W,W]^d$, with $W\geq2 \pi$. By Proposition \ref{tensor:prop1}, we parameterize the sequence of positive eigenvalues of $\cT_{Q,S}$ by $\lambda_{\vec{\ell}} := \prod_{k=1}^d \lambda_{\ell_k}$ for $\vec{\ell} = (\ell_1,\cdots,\ell_d) \in \N^d$, where $(\lambda_\ell)_{\ell \in \N}$ is the sequence of positive eigenvalues of the one-dimensional operator $\cT_{[0,1],[-W,W]} : L^2(\R) \rightarrow L^2(\R)$, with $\lambda_\ell \rightarrow 0$ as $\ell \rightarrow \infty$.

It remains to estimate
\[
\begin{aligned}
& M_\epsilon(Q,S) := \# \left\{ (\ell_1,\cdots,\ell_d) \in \N^d : \prod_{k=1}^d \lambda_{\ell_k}  > \epsilon \right\}, \epsilon \in (0,1), \\
&N_\epsilon(Q,S) := \# \left\{ (\ell_1,\cdots,\ell_d) \in \N^d : \prod_{k=1}^d \lambda_{\ell_k}  \in (\epsilon, 1-\epsilon) \right\}, \epsilon \in (0,1/2).
\end{aligned}
\]
For $\epsilon \in (0,1/2)$, $N_\epsilon(Q,S) \leq M_\epsilon(Q,S) - M_{1-\epsilon}(Q,S)$. So, it will suffice to estimate $M_\epsilon(Q,S)$ for $\epsilon \in (0,1)$.

Given $\lambda_1,\cdots \lambda_d \in [0,1]$ and $\epsilon \in (0,1)$, if $\prod_{k=1
}^d\lambda_k > \epsilon$ then $\lambda_k > \epsilon$ for all $k$.  Conversely, if $\lambda_k > \epsilon^{1/d}$ for all $k$, then $\prod_{k=1
}^d \lambda_k > \epsilon$. Therefore,
\[
\left \{ \ell: \lambda_\ell > \epsilon^{1/d} \right\}^d \subset \left\{ (\ell_1,\cdots,\ell_d) : \prod_{k=1}^d \lambda_{\ell_k}  > \epsilon \right\} \subset \left\{ \ell : \lambda_\ell > \epsilon \right\}^d.
\]
So,
\begin{equation}\label{tensor_bd:eqn}
( \# \{ \ell : \lambda_\ell > \epsilon^{1/d}\} )^d \leq M_\epsilon(Q,S) \leq ( \# \{ \ell : \lambda_\ell > \epsilon\})^d.
\end{equation}

By Proposition \ref{1d:prop}, 
\[
\begin{aligned}
 W/\pi - C \log(W) \log(1/(\epsilon(1 -\epsilon))) &\leq  \# \{ \ell : \lambda_\ell > \epsilon \} \\
&\leq W /\pi + C \log(W) \log(1/(\epsilon(1-\epsilon))).
\end{aligned}
\]
We take both sides to the power $d$, and apply the Binomial Theorem. Letting $\chi(\epsilon) = \log(1/(\epsilon(1-\epsilon)))$,
\[
\begin{aligned}
& (W/\pi)^d - \sum_{j=0}^{d-1} \binom{d}{j} (W/\pi)^j (C \log(W) \chi(\epsilon))^{d-j} \\
&\leq (\# \{ \ell : \lambda_\ell > \epsilon \})^d \leq (W/\pi)^d + \sum_{j=0}^{d-1} \binom{d}{j} (W/\pi)^j (C \log(W) \chi(\epsilon))^{d-j}.
\end{aligned}
\]
Note $\binom{d}{j} \leq 2^d$, so the above series is bounded by $(2C)^d \sum_{j=0}^{d-1} W^j (\log(W) \chi(\epsilon))^{d-j}$. This series has monotone terms, so its value is at most $d$ times the maximum of the first and last terms. Thus, 
\[
\left| \# (\{ \ell : \lambda_\ell > \epsilon \})^d - (W/\pi)^d \right| \lesssim \max \left\{ W^{d-1} \log(W) \chi(\epsilon), \left(\log(W)  \chi(\epsilon)\right)^d \right\}.
\]
We apply the previous inequality with $\epsilon$ replaced by $\epsilon^{1/d}$, to get:
\[
\left| \# (\{ \ell : \lambda_\ell > \epsilon^{1/d} \})^d - (W/\pi)^d \right| \lesssim \max \{ W^{d-1} \log(W) \chi(\epsilon^{1/d}) ), (\log(W) \chi(\epsilon^{1/d}))^d \}.
\]
Note that $\epsilon^{1/d} \leq 1 + \frac{1}{d}(\epsilon - 1)$ for $\epsilon \in (0,1)$.\footnote{The function $t \mapsto t^{1/d}$ is concave on $(0,\infty)$, and its first order taylor polynomial at $t=1$ is $1 + \frac{1}{d}(t-1)$; thus, by concavity, $t^{1/d} \leq 1 + \frac{1}{d} (t-1)$ for $t>0$.} Thus, $1- \epsilon^{1/d} \geq \frac{1}{d} (1-\epsilon)$, so
\[
\begin{aligned}
\chi(\epsilon^{1/d}) &= \log(1/(\epsilon^{1/d} (1-\epsilon^{1/d}))) = \log(1/\epsilon^{1/d}) + \log(1/(1-\epsilon^{1/d})) \\
&\leq \log(1/\epsilon) + \log(d/(1-\epsilon)) = \log(d/(\epsilon(1-\epsilon))) \lesssim  \chi(\epsilon).
\end{aligned}
\]
Using the previous inequalities in \eqref{tensor_bd:eqn}, for any $\epsilon \in (0,1)$,
\begin{equation}\label{tensor_bd2:eqn}
\begin{aligned}
|M_\epsilon(Q,S) - (W/\pi)^d|  \lesssim \max \{ W^{d-1} \log(W) \chi(\epsilon) , (\log(W) \chi(\epsilon))^d \}.
\end{aligned}
\end{equation}
When $\epsilon \in (0,1/2)$, we have $1/(1-\epsilon) \leq 2$, so $\chi(\epsilon) = \log(1/(\epsilon(1-\epsilon))) \leq \log(2/\epsilon) \lesssim \log(1/\epsilon)$, so, \eqref{tensor_bd2:eqn} implies the first line of \eqref{M_eps1:eqn}.

Next, we give the bound on $N_\epsilon(Q,S)$. For $\epsilon \in (0,1/2)$,
\[
\begin{aligned}
N_\epsilon(Q,S) &\leq M_\epsilon(Q,S) - M_{1-\epsilon}(Q,S) \\
& \leq | M_\epsilon(Q,S) - (W/\pi)^d| + |M_{1-\epsilon}(Q,S) - (W/\pi)^d|.
\end{aligned}
\]
Now, apply \eqref{tensor_bd2:eqn} (note the RHS is unchanged if we replace $\epsilon$ by $1 - \epsilon$), to give
\[
N_\epsilon(Q,S) \lesssim \max \{ W^{d-1} \log(W) \chi(\epsilon), (\log(W) \chi(\epsilon))^d \}.
\]
As before, if $\epsilon \in (0,1/2)$, we have $\chi(\epsilon) \lesssim \log(1/\epsilon)$, so the above inequality implies the second line of \eqref{M_eps1:eqn}.

This completes the proof of Theorem \ref{tensor:propA}.

\subsection{Proof of Corollary \ref{eig_decay:cor} }\label{decay-rate:eign}

% We denote by $\lambda_k(Q,S)$ the eigenvalues of the operator $\cT_{Q,S} =  P_Q B_Q P_Q$, and by 
%$\lambda_k(S,Q)$  the the eigenvalues of the operator $\cT_{S,Q} =  B_S P_Q B_S$.  {\color{red} Whoops! This notation is problematic -- $\cT_{Q,S}$ and $\cT_{S,Q}$ use space/frequency projections in different orders. In particular, the range of $\cT_{Q,S}$ is spacelimited functions and the range of $\cT_{S,Q}$ is bandlimited functions. Need $\cT_{X,Y}$ to mean a fixed thing for the entire paper. Which order to use?}
%{\color{blue} shall we use $\widetilde{\cT}$ for order BPB?}

 Given measurable sets $Q,S$ in $\R^d$ with positive and finite $d$-dimensional Lebesgue measure, let $\{\lambda_k(Q,S)\}_{k \geq 0}$ be the sequence of nonzero eigenvalues of the compact operator $\cT_{Q,S} = P_Q B_S P_Q$ on $L^2(\R^d)$, counted with multiplicity, and arranged in non-increasing order.
 
The following observations are due to Landau; see Lemma 1 of \cite{Landau67}. 

\begin{lemma}\label{observation:lem}
For all $k \geq 0$,
\begin{align}
&\lambda_k(Q,S) = \lambda_k(S,Q),  \label{eqn:eval-sym}\\
&\lambda_k(Q,S_1) \leq \lambda_k(Q,S_2), \quad \mbox{ if } S_1 \subset S_2, \label{eqn:eval-mon1} \\
&\lambda_k(Q_1,S) \leq \lambda_k(Q_2,S), \quad \mbox{ if } Q_1 \subset Q_2.  \label{eqn:eval-mon2}
\end{align}
\end{lemma} 
  
%\begin{proof} 
%For the proof of \eqref{eqn:eval-sym} and  \eqref{eqn:eval-mon1} see Lemma 1 of \cite{Landau67}. Clearly, \eqref{eqn:eval-mon2} is a consequence of  \eqref{eqn:eval-sym} and  \eqref{eqn:eval-mon1}. 
%follows from the similarity of the operators   $\widetilde{\cT}_{Q,S}$  and $\widetilde{\cT}_{S,Q}$.  Indeed: xxx.. this proves the claim. 

  %  \footnote{For the proof of \eqref{eqn:eval-sym}, note that $\cT_{Q,S} = P_Q B_S P_Q = L_{Q,S}^* L_{Q,S}$, where $L_{Q,S} = B_S P_Q$. Also, $\cT_{S,Q} = L_{Q,S} L_{Q,S}^*$. By well-known results, the positive eigenvalues of $L L^*$ and $L^* L$ are equal for any operator $L$ on a Hilbert space (both are equal to the squared singular values of $L$). \\
    
  %  Note the bandlimiting operators satisfy that $B_{S_2} - B_{S_1} = B_{S_2 \setminus S_1}$ if $S_1 \subset S_2$. Then $\cT_{Q,S_2} - \cT_{Q,S_1} = P_Q B_{S_2} P_Q - P_Q B_{S_1} P_Q = P_Q B_{S_2 \setminus S_1} P_Q = \cT_{Q, S_2 \setminus S_1}$ is a positive semidefinite operator.  Thus, by the variational (max-min) characterization of eigenvalues, one has $\lambda_k(Q,S_1) \leq  \lambda_k(Q,S_2)$, proving \eqref{eqn:eval-mon1}. Finally, \eqref{eqn:eval-mon2} follows from \eqref{eqn:eval-sym}, \eqref{eqn:eval-mon1}. 
%  \end{proof}

Now, we are ready to prove Corollary \ref{eig_decay:cor}. 

\begin{proof}[Proof of Corollary \ref{eig_decay:cor}]

Let $\Delta = \diam_\infty(Q) \diam_\infty(S) < \infty$. Rescalings $(Q,S) \mapsto (\kappa^{-1} Q, \kappa S)$  preserve the eigenvalues of SSLOs. Similarly, translations of the domains $Q$ and $S$ preserve the eigenvalues of SSLOs. So, without loss of generality, $Q \subset [0,1]^d$ and $S \subset [-\Delta,\Delta]^d$. We may also assume  that $\Delta > 1$. By Lemma \ref{observation:lem},
\[
\lambda_k(Q,S) \leq \lambda_k([0,1]^d,[-\Delta,\Delta]^d).
\]
Therefore, to prove the result, it suffices to show that
\[
\lambda_k(\Delta) := \lambda_k([0,1]^d,[-\Delta,\Delta]^d) \leq C_d \exp( - c(\Delta) k^{1/d}), 
\]
for constants $C_d > 0$ (determined by $d$) and $c(\Delta) > 0$ (determined by $d$ and $\Delta$).

In \eqref{M_eps1:eqn} of Theorem \ref{tensor:propA}, we showed that, for any $\epsilon \in (0,1/2)$,
\[
\# \{ k: \lambda_k(\Delta) > \epsilon \} \leq (\Delta/\pi)^d + C_d \max\{ \Delta^{d-1} \log(\Delta) \log(1/\epsilon), (\log(\Delta) \log(1/\epsilon))^d \}.
\]
Since $\log(1/\epsilon)^d \geq \log(1/\epsilon) \geq 1$ for $\epsilon < 1/2$, and  $0 < \log(\Delta) \leq \Delta$ for $\Delta > 1$,
\begin{equation}
\label{eig_count1:eqn}
\# \{ k: \lambda_k(\Delta) > \epsilon \} \leq C_d \Delta^d  \log(1/\epsilon)^d \quad \mbox{for }  \epsilon \in (0,1/2).
\end{equation}

Note that $\# \{ k  \geq 0: \lambda_k(\Delta) > \epsilon \} \leq T \iff \lambda_{T}(\Delta) \leq \epsilon$ for any integer $T > 0$. 

Now, for an arbitrary integer $K >  C_d \Delta^d \log(2)^d$, choose $\epsilon \in (0,1/2)$ such that  $C_d \Delta^d \log(1/\epsilon)^d  \leq K < C_d \Delta^d \log(1/\epsilon)^d + 1$. By \eqref{eig_count1:eqn}, $\# \{ k : \lambda_k(\Delta) > \epsilon \} \leq K$, and so, $\lambda_{K}(\Delta) \leq \epsilon$. Note that $K \leq C_d \Delta^d \log(1/\epsilon)^d + 1$ implies that $\epsilon \leq \exp( - c_d  (K-1)^{1/d}/\Delta)$ where $c_d  = C_d^{-1/d}$. Therefore,
\[
\lambda_{K}(\Delta) \leq \exp( - c_d  (K-1)^{1/d}/\Delta) \mbox{ provided that } K > C_d \Delta^d \log(2)^d.
\]

Note, if $K \leq C_d \Delta^d \log(2)^d$ then $\lambda_{K}(\Delta) \leq 1 \leq C_d' \exp( - c_d  (K-1)^{1/d}/\Delta)$, provided that $C_d' \geq \exp(c_d C_d^{1/d} \log(2))$. 

Therefore, 
\[
\lambda_{K}(\Delta) \leq \max\{1, C_d'\} \exp( - c_d  (K-1)^{1/d}/\Delta) \mbox{ for all } K \geq 1.
\]
This completes the proof of Corollary \ref{eig_decay:cor}.
\end{proof}

\section{Appendix: Properties of Gevrey functions}
\label{appendix:sec}

\begin{proof}[Proof of Proposition \ref{prop:Gevrey1}.]
By the product rule, $(fg)^{(k)}$ is the sum of $2^k$ terms of the form $f^{(k_1)} g^{(k_2)}$ with $k_1 + k_2 = k$. By hypothesis on $f$ and $g$,
\[
\begin{aligned}
\| f^{(k_1)} g^{(k_2)}\|_\infty \leq C R^{k_1} (k_1!)^r \cdot D S^{k_2} (k_2!)^r &\leq CD \max\{R,S\}^{k_1+k_2} ((k_1+k_2)!)^{r} \\
&= CD \max\{R,S\}^{k} (k!)^r.
\end{aligned}
\]
Therefore, by the triangle inequality,
\[
\| (fg)^{(k)} \|_\infty \leq 2^k CD \max\{R,S\}^{k} (k!)^r.
\]
So, $fg \in G^r_0(CD,2 \max\{R,S\})$, as claimed.

The below proof of the composition law for Gevrey functions appears in \cite{Figuer1} (see Theorem 2.5). We repeat the proof here for the sake of completeness.

Let $f : I \rightarrow J$ and $g : J \rightarrow \R$  be in $G_0^r(C,R)$ and $G_0^r(D,S)$, respectively. Apply the Fa\'{a} di Bruno formula for the $n$th derivative of $h = g \circ f$, $n\geq 1$, $t \in I$,
\[
h^{(n)}(t) = \sum \frac{n!}{k_1! \cdots k_n!} g^{(k)}(f(t)) \prod_{j=1}^n \left(\frac{f^{(j)}(t)}{j!} \right)^{k_j},
\]
where the sum is taken over all $k_1,\cdots,k_n$ with
\[
k_1 + 2k_2 + \cdots + n k_n = n,
\]
and $k := k_1 + \cdots + k_n$ varies in the summation.
Because $f \in G^r_0(C,R)$ and $g \in G^r_0(D,S)$, we learn that
\begin{equation}\label{eqn:main1}
\begin{aligned}
|h^{(n)}(t)| &\leq \sum \frac{n!}{k_1! \cdots k_n!} \left(D S^k (k!)^r\right) \prod_{j=1}^n \left(\frac{ C R^j (j!)^r}{j!} \right)^{k_j} \\
& = n! D R^n  \sum \frac{1}{k_1! \cdots k_n!} C^k S^k (k!)^r m(k_1,\cdots,k_n) ,
\end{aligned}
\end{equation}
where 
\[
m(k_1,\cdots,k_n) = \prod_{j=1}^n \left( (j!)^{r-1} \right)^{k_j}.
\]
The sequence $(n!)^{1/(n-1)}$ is increasing. Therefore,
\[
(m!)^{r-1} \leq ((n!)^{r-1})^{(m-1)/(n-1)}
\]
for $1 < m \leq n$. Therefore,
\[
m (k_1,\cdots,k_n) = \prod_{j=1}^n \left( (j!)^{r-1} \right)^{k_j} \leq ((n!)^{r-1})^{(k_2 + 2k_3 + \cdots (n-1)k_n)/(n-1)}.
\]
and
\[
(k!)^r = k! (k!)^{r-1} \leq k! ((n!)^{r-1})^{(k-1)/(n-1)},
\]
so
\[
(k!)^r m(k_1,\cdots,k_n)  \leq  k! ((n!)^{r-1})^{(k_2 + 2k_3 + \cdots = (n-1)k_n + k-1)/(n-1)} = k! (n!)^{r-1}.
\]
Plugging that into \eqref{eqn:main1}, we get
\[
|h^{(n)}(t)| \leq D R^n (n!)^r \sum \frac{k!}{k_1!\cdots k_n!} C^k S^k
\]
Applying the following equality (see Lemma 1.3.2 of \cite{Krantz})
\[
\sum \frac{k!}{k_1!\cdots k_n!} H^k = H (1+H)^{n-1},
\]
gives
\[
|h^{(n)}(t)| \leq D R^n (n!)^r CS (1+CS)^{n-1} \leq D (R(1+CS))^n (n!)^r.
\]
This is valid for any $n \geq 1$.

Obviously $h = g \circ f$ satisfies $\| h \|_\infty \leq \| g \|_\infty \leq D$. Therefore, the above inequality is valid also for $n=0$, which implies that $g \circ f \in G^r_0(D,T)$ for $T = R(1+CS)$.

\end{proof}

\begin{proof}[Proof of Proposition \ref{prop:Gevrey_PW}]
We have $|\widehat{B}(\xi)| \leq \int_{-2}^2 |B(x)| dx \leq 4 C$. So, in particular, the bound is true for $\xi = 0$.

Now, for $\xi \neq 0$, apply integration by parts, and condition (a), and get that for any integer $k \geq 0$,
\[
\widehat{B}(\xi) = (-i\xi)^{-k} \int_{-2}^2 e^{-i x \xi} B^{(k)}(x) dx
\]
Then using condition (b),
\[
| \widehat{B}(\xi)| \leq |\xi|^{-k} 4 C R^k (k!)^{3/2} \leq 4 C (R/|\xi|)^k k^{3k/2}.
\]
In particular, $|\widehat{B}(\xi)| \leq 4C$. Further, by geometric averaging of the previous formula at consecutive integers, we obtain for any real number $s > 0$,
\[
| \widehat{B}(\xi)| \leq  4 C (R/|\xi|)^s \lceil s \rceil^{3\lceil s \rceil /2}.
\]
In particular, we can take $s = 2n/3$ for an integer $n \geq 1$. Note that $\lceil s \rceil \leq \min \{ n, 2n/3 + 1\}$, so
\[
| \widehat{B}(\xi)| \leq  4 C (R/|\xi|)^{2n/3} n^{2n/3 + 1}, \qquad n \geq 1.
\]
Now,
\[
|\widehat{B}(\xi)| \exp(a |\xi|^{2/3}) =  \sum_{n=0}^\infty \frac{a^n |\xi|^{2n/3} |\widehat{B}(\xi)|}{n!} = |\widehat{B}(\xi)| + \sum_{n=1}^\infty \frac{a^n |\xi|^{2n/3} |\widehat{B}(\xi)|}{n!}.
\]
Using the Stirling approximation, $n! \geq (n/e)^n$, and the previous bounds, we get
\[
|\widehat{B}(\xi)| \exp(a |\xi|^{2/3}) \leq 4C + 4C  \sum_{n=1}^\infty (ea R^{2/3})^n n.
\]
Now, set $a = \frac{1}{2 e R^{2/3}}$, so we get
\[
|\widehat{B}(\xi)| \exp(a |\xi|^{2/3}) \leq 4C +  4C  \sum_{n=1}^\infty 2^{-n} n = 12C.
\]
\end{proof}

\addcontentsline{toc}{section}{References}

\printbibliography

\end{document}